\documentclass[10.5pt,a4paper]{article}

\usepackage{graphicx,latexsym,euscript,makeidx,color,bm}
\usepackage{amsmath,amsfonts,amssymb,amsthm,thmtools,mathtools,mathrsfs,enumerate}
\usepackage[colorlinks,linkcolor=blue,anchorcolor=green,citecolor=red]{hyperref}

\usepackage{tikz}
\usetikzlibrary{fit,calc,positioning} 
\tikzstyle{Block}=[rectangle,minimum width=3cm,minimum height=1cm,text centered,text width=5.2cm,draw=black]
\tikzstyle{Implication}=[rectangle,minimum width=3cm,minimum height=1cm,text centered,text width=5.2cm]
\tikzstyle{jian}=[<->, >=stealth]

\usepackage{geometry}
\allowdisplaybreaks[4]

\def\5n{\negthinspace \negthinspace \negthinspace \negthinspace \negthinspace }
\def\4n{\negthinspace \negthinspace \negthinspace \negthinspace }
\def\3n{\negthinspace \negthinspace \negthinspace }
\def\2n{\negthinspace \negthinspace }
\def\1n{\negthinspace }

   \def\cD{{\cal D}}  
\def\dbE{\mathbb{E}}     
\def\dbF{\mathbb{F}} \def\sF{\mathscr{F}}    
         
\def\dbH{\mathbb{H}}

   \def\cL{{\cal L}}  
   \def\cM{{\cal M}} \def\BM{{\bf M}} 
   \def\cN{{\cal N}} \def\BN{{\bf N}} 
     
\def\dbP{\mathbb{P}}     
     
\def\dbR{\mathbb{R}}     
\def\dbS{\mathbb{S}}   \def\cS{{\cal S}}  
   \def\cT{{\cal T}}  
   \def\cU{{\cal U}} \def\BU{{\bf U}}

   \def\cX{{\cal X}} \def\BX{{\bf X}} 
   \def\cY{{\cal Y}} \def\BY{{\bf Y}} 
   \def\cZ{{\cal Z}} \def\BZ{{\bf Z}} 
\def\Om{\Omega}           \def\Th{\Theta}

\def\ss{\smallskip}                
\def\ms{\medskip}                
               \def\lan{\langle}
\def\ds{\displaystyle}           \def\ran{\rangle}

\def\no{\noindent}        \def\q{\quad}                      \def\llan{\left\langle}
\def\ns{\noalign{\ss}}    \def\qq{\qquad}                    \def\rran{\right\rangle}
    \def\hb{\hbox}                     
                   
         \def\rf{\eqref}                    
  \def\deq{\triangleq}               
            \def\({\Big (}
\def\les{\leqslant}                  \def\){\Big )}
       \def\geq{\geqslant}
\def\ges{\geqslant}       \def\esssup{\mathop{\rm esssup}}   \def\[{\Big[}
        \def\essinf{\mathop{\rm essinf}}   \def\]{\Big]}
          \def\tr{\hbox{\rm tr$\,$}}         \def\cd{\cdot}
\def\wt{\widetilde}              
                            
\def\nn{\nonumber}        \def\ts{\times}                    \def\vf{\vartheta}

\def\a{\alpha}        \def\G{\Gamma}      \def\O{\Omega}   \def\o{\omega}
\def\b{\beta}         \def\D{\Delta}   \def\d{\delta}        
\def\z{\zeta}         \def\Th{\Theta}      
\def\e{\varepsilon}   \def\L{\Lambda}  \def\l{\lambda}        
    \def\t{\tau}       \def\i{\infty}   

\def\bde{\begin{definition}\label}    \def\ede{\end{definition}}
\def\be{\begin{equation}}
	\def\bel{\begin{equation}\label}      \def\ee{\end{equation}}
	\def\bt{\begin{theorem}\label}        \def\et{\end{theorem}}
	\def\bc{\begin{corollary}\label}      \def\ec{\end{corollary}}
	\def\bl{\begin{lemma}\label}          \def\el{\end{lemma}}
	\def\bp{\begin{proposition}\label}    \def\ep{\end{proposition}}
	\def\bas{\begin{assumption}\label}    \def\eas{\end{assumption}}
	\def\br{\begin{remark}\label}         \def\er{\end{remark}}
	\def\bex{\begin{example}\label}       \def\ex{\end{example}}
	\def\ba{\begin{array}}                \def\ea{\end{array}}
	\def\ben{\begin{enumerate}}           \def\een{\end{enumerate}}
	
	\newtheorem{theorem}{Theorem}[section]
	\newtheorem{definition}[theorem]{Definition}
	\newtheorem{proposition}[theorem]{Proposition}
	\newtheorem{corollary}[theorem]{Corollary}
	\newtheorem{lemma}[theorem]{Lemma}
	\newtheorem{remark}[theorem]{Remark}
	\newtheorem{example}[theorem]{Example}
	
	\makeatletter
	
	\@addtoreset{equation}{section}
	\makeatother
	
	\allowdisplaybreaks

\begin{document}

		\title{{\bf Stochastic Linear-Quadratic Optimal Control Problems with Random Coefficients and Markovian Regime Switching System}}

		\author{
			Jiaqiang Wen\thanks{Department of Mathematics and SUSTech International Center for Mathematics, Southern University of Science and Technology,
				Shenzhen, Guangdong, 518055, China (Email: {\tt wenjq@sustech.edu.cn}).
				This author is supported by National Natural Science Foundation of China (grant No. 12101291) and
				Guangdong Basic and Applied Basic Research Foundation (grant No. 2022A1515012017).}~,~~~
			Xun Li\thanks{Department of Applied Mathematics, The Hong Kong Polytechnic University,
				Hong Kong, China (Email: {\tt li.xun@polyu.edu.hk}).
				This author is supported by RGC of Hong Kong (grant Nos. 15209614,
				15213218 and 15215319) and partially from CAS AMSS-PolyU Joint Laboratory of Applied Mathematics.}~,~~~
			Jie Xiong\thanks{Department of Mathematics and SUSTech International center for Mathematics,
				Southern University of Science and Technology,
				Shenzhen, Guangdong, 518055, China (Email: {\tt xiongj@sustech.edu.cn}).
				This author is supported by SUSTech Start up fund Y01286120 and National Natural Science Foundation of China (grant No. 61873325).}~,~~~
			Xin Zhang\thanks{School of Mathematics, Southeast University, Nanjing, Jiangsu, 211189, China (Email: {\tt x.zhang.seu@gmail.com}).
				This author is supported by National Natural Science Foundation of China (grant No. 12171086) and Fundamental Research Funds for the Central Universities (grant No. 2242021R41082).}
		}
		
		\date{}
		
		\maketitle

		
		\no\bf Abstract. \rm
		This paper thoroughly investigates stochastic linear-quadratic optimal control problems with the Markovian regime switching system, where the coefficients of the state equation and the weighting matrices of the cost functional are random.
		We prove the solvability of stochastic Riccati equation under the uniform convexity condition, and obtain the closed-loop representation of the open-loop optimal control using the unique solvability of the corresponding stochastic Riccati equation.
		Moreover, by applying It\^{o}'s formula with jumps, we get a representation of the cost functional on a Hilbert space, characterized as the adapted solutions of some forward-backward stochastic differential equations.
		We show that the necessary condition of the open-loop optimal control is the convexity of the cost functional, and the sufficient condition of the open-loop optimal control is the uniform convexity of the cost functional. In addition, we study the properties of the stochastic value flow of the stochastic linear-quadratic optimal control problem.
		Finally, as an application, we present a continuous-time mean-variance portfolio selection problem and prove its unique solvability.

		\ms
		
		\no\bf Key words: \rm Stochastic linear-quadratic optimal control, Markovian regime switching, random coefficient, stochastic Riccati equation, mean-variance portfolio selection.
		
		\ms
		
		\no\bf AMS subject classifications. \rm 49N10, 93E20.

		\section{Introduction}
		
		Linear-quadratic (LQ, for short) optimal control problem plays a fundamental role in control theory, which appeared with the birth of stochastic analysis and developed rapidly in recent decades due to its wide range of applications. The study of stochastic linear-quadratic (SLQ, for short) optimal control problems can be traced back to the work of Kushner \cite{Kushner-62} and Wonham \cite{Wonham 1968}.
		In the classical setting, the SLQ optimal control problem can be solved elegantly via the Riccati equation under some mild conditions on the weighting coefficients (see Yong--Zhou \cite[Chapter 6]{Yong-Zhou 1999}).
		Chen--Li--Zhou \cite{Chen-Li-Zhou 1998} investigated SLQ optimal control problems with indefinite weighting control matrix and their applications in solving continuous-time mean-variance portfolio selection problems in 1990s. From then on, there has been increasing interest in the so-called {\it indefinite} SLQ optimal control problems as well as addressing their applications (see Chen--Yong \cite{Chen-Yong-01}, Ait Rami--Moore--Zhou \cite{Ait Rami-Moore-Zhou 2001}, and Li--Zhou--Lim \cite{Li-Zhou-Lim 2002}).

		\ms
		
		To tackle SLQ optimal control problems with random coefficients, Bismut \cite{Bismut 1976} initially derived the existence and uniqueness of the solution of the stochastic Riccati equation (SRE, for short) using techniques of functional analysis. Peng \cite{Peng 1999} posted it as an open problem in the general setting in his collection of open problems for backward stochastic differential equations (BSDEs, for short). %
		Kohlmann--Tang \cite{Kohlmann-Tang 2003-2}, Tang \cite{Tang 2003} and Sun--Xiong--Yong \cite{Sun-Xiong-Yong-21} established the existence and uniqueness of solutions to the related SRE under different conditions using their methods, which partially addressed the open problem introduced by Peng \cite{Peng 1999}.
		In detail, Kohlmann--Tang \cite{Kohlmann-Tang 2003-2} studied the multidimensional backward stochastic Riccati equations and gave an application to the stochastic optimal control problem.
		Tang \cite{Tang 2003} studied general SLQ optimal control problems with random coefficients and proved the existence and uniqueness of related backward stochastic Riccati equations.
		Sun--Xiong--Yong \cite{Sun-Xiong-Yong-21} studied the SLQ  optimal control problems with random coefficients, proved the solvability of the corresponding SRE, and obtained the closed-loop representation of the open-loop optimal control.
		Li--Wu--Yu \cite{Li-Wu-Yu 2018} analyzed a special type of indefinite SLQ problem with random coefficients. For more details about the efforts devoted to the stochastic Riccati equation and its connection with SLQ optimal control problems, we refer the interested readers to Kohlmann--Tang \cite{Kohlmann-Tang 2001}, Tang \cite{Tang 2015}, etc.

		\ms		
		
		Recently, there has been a dramatically increasing interest in the SLQ optimal control problems with random jumps, such as Poisson jumps or the regime switching jumps, which are of practical importance in various fields such as economics, financial management, science, and engineering. In the past few years, researchers have focused on models of financial markets whose key parameters are described by Markov processes, such as stock returns, interest rates, and volatility.
		In particular, one could face two market regimes in financial markets, one of which stands for a bull market with price rises, while the other for a bear market with price drops. We call such a formulation the regime switching model, where the market parameters depend on market modes that switch among a finite number of regimes.
		More recently, applications of SLQ optimal control problems with regime switching models or Poisson jumps have been extensively developed. For instance, Ji--Chizeck \cite{Ji-Chizeck 1990} formulated a class of continuous-time LQ optimal control problems with Markovian jumps.
		Zhou--Yin \cite{Zhou-Yin 2003} studied a mean-variance portfolio selection with regime switching.
		Liu--Yin--Zhou \cite{Liu-Yin-Zhou 2005} studied the near-optimal controls of regime switching LQ control problems.
		Hu--Oksendal \cite{Hu-Oksendal 2008} discussed SLQ optimal control problems with Poisson jumps and partial information using the technique of completing squares.
		%
		%
		Song--Tang--Wu \cite{Song-Tang-Wu 2020} established the maximum principle for progressive stochastic optimal control problems with random jumps.
		Zhang--Li--Xiong \cite{Zhang-Li-Xiong 2021} investigated open-loop and closed-loop solvabilities for SLQ optimal control problems with a Markovian regime switching system.
		Hu--Shi--Xu \cite{Hu-Shi-Xu 2021} applied the constrained SLQ control with regime switching to a portfolio problem.
		For some other important works, we refer the readers to
		\cite{Donnelly-Heunis 2012,Li-Zhou-Rami 2003,Wen-Li-Xiong 2021,Zhang-Elliott-Siu 2012}, and the references therein.

		\ms
		
		In a real market, besides the Markov chain, it is more reasonable allowing the market parameters to depend on the Brownian motion, due to the fact that the interest rates, stock rates, and volatilities are affected by the uncertainties caused by the Brownian motion.
		However, up to now, few results have been obtained on this topic.
		In this paper, inspired by the continuous-time mean-variance portfolio selection problems, we are interested in studying this topic, i.e., the SLQ optimal control problems under the Markovian regime switching system with random coefficients.
		Further, by developing some ideas of Sun--Xiong--Yong \cite{Sun-Xiong-Yong-21} and Li--Zhou--Lim \cite{Li-Zhou-Lim 2002}, we establish the existence of an open-loop optimal control and prove the unique solvability of the associated SRE.
		Also, we apply the theoretical results shown in this paper to treat a continuous-time mean-variance portfolio selection problem.
		Next, we present our main results and difficulties in detail.
		\begin{enumerate}
			\item[(i)] We first introduce the associated stochastic Riccati equation of Problem (M-SLQ) (see \rf{optim-2}), and prove that the optimal state process of Problem (M-SLQ) is invertible under the uniform convexity condition \rf{21.8.21.3.2}. Then, we show that a bounded process $\hat P(\cd,\a(\cd))$, together with two square-integrable processes $\hat \Lambda(\cd)$ and $\hat \z(\cd)$, uniquely satisfies the associated stochastic Riccati equation \rf{21.8.23.6}, and thus establishing the unique solvability of the associated SRE (see \autoref{21.9.3.1}). Moreover, we derive a closed-loop representation for the open-loop optimal control of Problem (M-SLQ) using the unique solvability of the SRE (see \autoref{21.11.26.1}).
			\item[(ii)] In order to prove the above theoretical results, we prove some auxiliary results. We establish the equivalence between Problem (M-SLQ) and Problem ${(\hb{M-SLQ})}_0$, i.e., a control $u^*(\cd)\in\cU[t,T]$ is optimal for Problem (M-SLQ)$_0$ if and only if it is optimal for Problem ${(\hb{M-SLQ})}$ (see \autoref{Theorem4.2}), and based on the equivalence, we analyze the time-consistency of the optimal control (see \autoref{Corollary4.5}).
			Then, we obtain a quadratic representation of the stochastic value flow in terms of a bounded, left continuous, and $\mathbb{S}^{n}$-valued process (see \autoref{21.8.21.5} and \autoref{21.8.22.4}).
			In addition, by the technique of It\^{o}'s formula with jumps, we represent the cost functional of Problem (M-SLQ) as a bilinear form, in terms of the adapted solutions of some forward-backward stochastic differential equations (FBSDEs, for short) in a suitable Hilbert space (see \autoref{Theorem3.4}).
			\item[(iii)] As a financial application, we present an example of the continuous-time mean-variance portfolio selection problem and prove the unique solvability of the related mean-variance problem. Also, we derive the representation of the unique optimal investment strategy (see \autoref{22.6.20.1}), which further develops the work of Li--Zhou--Lim \cite{Li-Zhou-Lim 2002} to the Markovian regime switching system with random coefficients.
		\end{enumerate}
		Compared with Sun--Xiong--Yong \cite{Sun-Xiong-Yong-21}, the difficulties of this paper mainly come from the solvability of the associated SRE, due to the presence of the Markovian regime switching jumps.
		\begin{enumerate}
			\item [(iv)] Firstly, due to the presence of random coefficients, the Riccati equation associated with Problem (M-SLQ) becomes a nonlinear BSDE, usually referred to as the backward stochastic Riccati equation.  Furthermore, the backward stochastic Riccati equation is driven by both the Brownian motion and the martingales $(\widetilde{N}_{kl}(\cdot))_{k,l\in\cS}$ generated by the Markov chain thanks to the occurrence of Markovian regime switching jumps in the model. Thus, the solvability of this BSDE is more complicated than that of the model with deterministic coefficients and without the Markovian regime switching jumps.
			\item [(v)] Secondly, let  $(X_j(s), Y_j(s),Z_j(s),\Gamma^j(s))$ be the solution of the FBSDE composed by the open-loop optimal state process and adjoint equation corresponding to initial state $(t,e_j,\vartheta)$. It is worth to mention that the construction of the fourth component $\boldsymbol{\G}(s)$ for the solution of matrix-valued FBSDE \eqref{22.6.25.2} is not directly extending the dimensional by combing $\G^j(s)$ as $(\G^1(s),\cdots,\G^n(s))$. Furthermore, the $\G$-term of the solution to equation \eqref{21.8.20.1} associated with the initial state $(t,\xi,\vartheta)$ is represented as $\boldsymbol{\G}(s)\circ\xi$, which is different from that of $(Y(s), Z(s))=(\BY(s)\xi, \BZ(s)\xi)$. Please see more details in \autoref{21.8.21.2}. 
			
			\item [(vi)] Thirdly, the existence of  process $P(t)$ appeared in the equation (42) of Sun--Xiong--Yong \cite{Sun-Xiong-Yong-21} for the representation of the value function follows directly by its definition $\dbE[\BM(T)+\int_t^T\BN(s)ds|\sF_t]$. 
				In our case, we have a similar representation for the value function but with $\BM(T)$ and $\BN(s)$ replaced by functions depend on the Markov chain $\alpha(\cd)$ and $\sF_t$ replaced by the $\sigma$-field generated by the Brownian motion and the Markov chain $\alpha(\cd)$.
			To solve our problem, we need to prove that there exists a process  $P:[0,T]\ts\cS\ts\Omega\rightarrow\dbS^n$, which is $\dbF$-adapted, such that
			\begin{equation*}
				P(t,\a(t))= \mathbb{E}\Big[\mathbf{M}(T,\a(T))+\int_{t}^{T} \boldsymbol{N}(s,\a(s)) d s \Big| \sF_{t}\Big].
			\end{equation*}
			This cannot be obtained directly from the definition of conditional expectation and we give proof for this (see the proof of \autoref{21.8.21.5}). In our model, although the evolution of ${\bf X}(s)$ and ${\bf Y}(s)$ depends on the Markov chain $\alpha(\cd)$, ${\bf Y}(s) {\bf X}(s)^{-1}$ may not equal $P(s, \alpha(s))$. We further prove the equality of $P(s, \alpha(s))={\bf Y}(s) {\bf X}(s)^{-1}$ (see the proof of \autoref{21.9.3.1}), which plays an important role in solving our problem.

		\end{enumerate}

		The paper is organized as follows. In \autoref{Sec2}, we present some preliminaries and formulate Problem (M-SLQ) with random coefficients and regime switching. In  \autoref{MainResult}, we state our main results, i.e., the solvability of the corresponding stochastic Riccati equation, the closed-loop representation of the open-loop optimal control, and some auxiliary results. In  \autoref{ProofSRE}, we prove the invertibility of the optimal state process, the solvability of the corresponding SRE, and the closed-loop representation of the open-loop optimal control.
		In  \autoref{ProofPri}, we strictly prove some auxiliary results in detail.  In  \autoref{Example}, we present an example of the continuous-time mean-variance portfolio selection problem under the Markovian regime switching system with random coefficients.
		In  \autoref{Conclusion}, we conclude the results.

		\section{Preliminaries}\label{Sec2}
		
		Let $(\O,\sF, \dbF, \dbP)$ be a complete filtered probability space on which a standard one-dimensional Brownian motion $\{W(t)\}_{t\ges 0}$ and a continuous-time and finite-state Markov chain $\{\a(t)\}_{t\ges 0}$ are defined, where the processes $W(\cd)$ and $\a(\cd)$ are independent and
		$\dbF=\{\sF_t\}_{t\ges 0}$ is the natural filtration of them with $\sF_0$ containing all $\dbP$-null sets of $\sF$. Let $\dbF^W=\{\sF^W_t\}_{t\ges 0}$ be the filtration generated by $W(\cd)$ and  $\dbF^\a=\{\sF^\a_t\}_{t\ges 0}$ be the filtration generated by $\a(\cd)$.
		We identify the state space of the Markov chain $\a(\cd)$ with a finite set $\cS\deq\{1, 2 \dots, D\}$, where $D\in \mathbb{N}$. Furthermore, the generator of the Markov chain  $\a(\cd)$ under $\mathbb{P}$ is denoted by $\l(t)\deq[\l_{kl}(t)]_{k, l \in \cS}$, where
		$\l_{kl}(t)$ is the constant transition intensity of the Markov chain from state $k$ to state $l$
		at time $t$.  For each fixed $k,l \in\cS$, we let $N_{kl}(t)$ be the number of jumps from state $k$ into state $l$ up to time $t$ and set $	\tilde{\l}_{kl}(t)\deq\int^{t}_{0}\l_{kl}(s)I_{\{\a(s-)=k\}} d s$.
		Let $N(t)\deq(N_{kl}(t))_{k,l\in \cal{S}}$ and $\widetilde{N}(t)\deq(\widetilde{N}_{kl}(t))_{k,l\in \cal{S}}$,
		where $\widetilde{N}_{kk} (t)\equiv0$ and $\widetilde{N}_{kl} (t)= N_{kl}(t)-\tilde{\l}_{kl}(t)$ when $k\neq l$.
		
		
		\ms
		
		Let $T>0$ be a fixed terminal time.
		The trace of a square matrix $M$ is denoted by $\tr[M]$,
		the set of all $n\times n$ symmetric matrices is denoted by $\dbS^n$,
		the set of all $D\times D$ matrices $M\deq (M_{kl})$ with $M_{kl}\in\dbS^n$ is denoted by $\mathbb{M}_{D}(\dbS^n)$, and
		the set of all $D\times D$ matrices $M\deq (M_{kl})$ with $M_{kl}\in\mathbb{R}^{n\times m}$ is denoted by
		$\mathbb{M}_{D}(\mathbb{R}^{n\times m})$.
		For $\dbH=\dbR^n$, $\dbR^{n\ts m}$ or $\dbS^n$, denote by $L^2_{\sF}(\O;\dbH)$ (resp., $L^\i_{\sF}(\O;\dbH)$) the set of all $\sF$-measurable, $\mathbb{H}$-valued, and square integrable (resp., bounded) random variables. 
		Denote by $L_\dbF^2(t,T;\dbH)$ (resp., $L_\dbF^\i(t,T;\dbH)$) the set of all $\dbH$-valued, $\dbF$-progressive measurable stochastic processes $\phi(s)$ with $\dbE\int^T_t|\phi(s)|^2ds<\i$ (resp., $ \esssup_{s\in[t,T]}|\phi(s)|<\i$),
		and denote by $L_\dbF^2(\O;C([t,T];\dbH))$ the set of all $\dbH$-valued, $\dbF$-adapted and continuous processes $\phi(s)$ with $\dbE[\sup_{s\in[t,T]}|\phi(s)|^2]<\i$.
		Moreover, for $\mathbb{G}=\mathbb{M}_{D}(\mathbb{R}^{n\times m})$ or $\mathbb{M}_{D}(\mathbb{S}^{n})$, we denote by $L_\dbF^2(t,T;\mathbb{G})$ the set of all $\mathbb{G}$-valued $\dbF$-progressively measurable process $\phi(s)\deq(\phi_{kl}(s))\in\mathbb{G}$ with $\dbE\int^T_t\sum_{k,l=1}^D|\phi_{kl}(s)|^2\lambda_{kl}(s)I_{\{\a(s-)=k\}}ds<\i.$ 
		%
		%
		For $\eta(s)\deq (\eta_{kl}(s))\in \cM_D(\mathbb{R}^{n\times n})$,  we further  define
		$$\eta(s)\bullet d\widetilde{N}(s)
		\deq\sum_{k,l=1}^D\eta_{kl}(s)d\widetilde{N}_{kl}(s).$$
		
		We now introduce the following state equation, which is the controlled Markovian
		regime switching linear stochastic differential equation (SDE, for short) over a finite time horizon $[t,T]$:
		\begin{equation}\label{state}
			\hspace{-0.38cm}\left\{\begin{aligned}
				\ds	&dX(s)=\big[A(s,\a(s))X(s)+B(s,\a(s))u(s)\big]ds \\
				\ns\ds&\qq\qq +\big[C(s,\a(s))X(s)+D(s,\a(s))u(s)\big]dW(s), \q s\in[t,T],\\
				\ns\ds	& \ds X(t)=\xi,\q \a(t)=\vartheta,
			\end{aligned}\right.
		\end{equation}
		%
		where $A(t, \omega, i)$, $B(t, \omega, i)$, $C(t, \omega, i)$ and $D(t, \omega, i)$ are given $\sF^W_t$-measurable processes for each $i \in \mathcal{S}$.
		We call $(t,\xi,\vartheta)$ an {\it initial triple}, which comes from the following set:
		$$\mathcal{D}=\big\{(t,\xi,\vartheta)\ |\ t\in[0,T],\ \xi\in L^2_{\sF_t}(\Om;\dbR^n),\
		\vartheta\in L_{\sF_t^\alpha}^2(\Omega;\mathcal{S})\big\}.$$
		%
		%
		In the state equation \rf{state}, the process $u(\cd)$ comes from the control space
		$\cU[t,T]\deq L^2_\dbF(t,T;\dbR^m)$, 
		called the {\it control process}, and the solution $X(\cd)$ of \rf{state} is called the {\it state process} with $(t,\xi,\vf)$ and $u(\cd)$.  
		%
		Let us state our SLQ optimal control problem.
		
		\ms
		
		\noindent\bf Problem $(\hb{M-SLQ})$. \rm For any given initial triple $(t,\xi,\vf)\in\cD$, find a control $u^{*}(\cd)\in\mathcal{U}[t,T]$, such that
		\begin{align}\label{optim-2}
			{J}(t,\xi,\vf;u^{*}(\cd))=\essinf_{u(\cd)\in\cU[t,T]}{J}(t,\xi,\vf;u(\cd))
			\deq {V}(t,\xi,\vf),
		\end{align}
		where the cost functional is given as the following quadratic form
		\bel{cost}\ba{ll}
		\ds {J}(t,\xi,\vf;u(\cd))\deq\dbE_t\left[\big\lan G(\a(T))X(T),X(T)\big\ran
		+\int_t^T\llan\begin{pmatrix}Q(s,\a(s))&S(s,\a(s))^\top\\S(s,\a(s))&R(s,\a(s))\end{pmatrix}
		\begin{pmatrix}X(s)\\ u(s)\end{pmatrix},
		\begin{pmatrix}X(s)\\u(s)\end{pmatrix}\rran ds\right].
		\ea\ee
		Note that $\dbE_t[\cdot]\deq\dbE[\cdot|\sF_t]$ represents the conditional expectation with respect to (w.r.t., for short) $\sF_t$. 
		For the initial triple $(t, \xi,\vf)$, we call the control process $u^{*}(\cd)$ an {\it open-loop optimal control} of Problem (M-SLQ) if it satisfies \rf{optim-2}, call the corresponding state process $X^{*}(\cdot) \equiv X\left(\cdot ; t, \xi,\vf, u^{*}(\cd)\right)$ an {\it open-loop optimal state process}, and call the state-control pair $\left(X^{*}(\cd), u^{*}(\cd)\right)$ an {\it open-loop optimal pair}. We call $(t, \xi,\vf) \mapsto {V}(t, \xi,\vf)$ the {\it stochastic value flow} of Problem ${(\hb{M-SLQ})}$, due to that the space $L_{\sF_{t}}^{2}\left(\Omega ; \mathbb{R}^{n}\right)$ becomes larger when $t$ increases.

		\begin{remark}\label{prob-expectation}\rm
			In the cost functional \rf{cost}, when the conditional expectation $\dbE_t[\cdot]$  degenerates to the expectation $\dbE$, we denote the related problem, cost functional, and value function by {\bf Problem (M-SLQ)$_0$}, $J_0(t,\xi,\vf;u(\cd))$, and $V_0(t,\xi,\vf)$, respectively. Clearly, $V_0(t,\xi,\vf)=\dbE[V(t,\xi,\vf)]$.
		\end{remark}
		%
		%

		\begin{definition}\label{def-open} \rm
			Problem (M-SLQ) is said to be
			%
			({\it uniquely}) {\it open-loop solvable} for the initial triple $(t,\xi,\vf)\in\cD$ if there exists a (unique) $u^*(\cd)=u^*(\cd\ ;t,\xi,\vf)\in\cU[t,T]$ (depending on $(t,\xi,\vf)$) such that
			$$J(t,\xi,\vf;u^{*}(\cd))\les J(t,\xi,\vf;u(\cd)), \q  \hb{a.s.},\ \forall u(\cd)\in\cU[t,T],$$
			%
			%
			and said to be ({\it uniquely}) {\it open-loop solvable} if it is (uniquely) open-loop solvable for any initial triple $(t,\xi,\vf)\in\cD$.
		\end{definition}

		For the coefficients in the state equation \rf{state} and the weighting matrices in the cost functional
		\rf{cost}, we post the following assumption:
		\begin{itemize}
			\item [{\bf(H)}] For any choice of $i\in\cS$,
			$A(t, \omega, i)$, $C(t, \omega, i)\in L_{\dbF^W}^\i(0,T;\dbR^{n\times n})$,
			$B(t, \omega, i)$, $D(t, \omega, i)\in L_{\dbF^W}^\i(0,T;\dbR^{n\times m})$,
			$G(\o,i)\in L^\i_{\sF^W_T}(\O;\dbS^n)$, $Q(t, \omega, i)\in L^\i_{\dbF^W}(0,T;\dbS^n)$,
			$S(t, \omega, i)\in L^\i_{\dbF^W}(0,T;\dbR^{m\times n})$, and
			$R(t, \omega, i)\in L^\i_{\dbF^W}(0,T;\dbS^m)$.
		\end{itemize}
		Under the condition (\textbf{H}), for any initial triple $(t,\xi,\vf) \in \mathcal{D}$ and any control $u(\cd) \in \mathcal{U}[t, T]$, the classical theory of SDEs (see Lemma 2.1 of Wen--Li--Xiong \cite{Wen-Li-Xiong 2021}) implies that the equation \rf{state} has a unique solution $X(\cdot) \equiv X(\cdot ; t, \xi,\vf, u(\cd))$, which is square-integrable and whose path is continuous.
			%
			Moreover, to simplify our further analysis, we finally introduce the following BSDE: 
			\begin{equation}\label{2.3}
				\left\{ \begin{aligned}
					\ds dM(s)=& -\big[M(s)A(s,\a(s))+A(s,\a(s))^\top M(s)+C(s,\a(s))^\top M(s)C(s,\a(s))+\Phi(s)C(s,\a(s))\\
					\ns\ds & +C(s,\a(s))^\top \Phi(s)+Q(s,\a(s))\big]ds
					+\Phi(s)dW(s)+\eta(s)\bullet d\wt{N}(s),\q  s\in[t,T],\\
					\ns\ds M(T)=&\ G(\a(T)),\q\a(t)=\vf.
				\end{aligned}\right.
			\end{equation}
			The solution of the above BSDE is denoted by the triple $(M(\cd),\Phi(\cd),\eta(\cd))$, where $\eta(\cdot)\deq (\eta_{kl}(\cdot))_{k,l\in\mathcal S} \in \cM_D(\mathbb{R}^{n\times n})$. 
			Note that the terminal value $G(\cd)$ is bounded, so the classical theory of BSDEs combining It\^o's formula with jumps deduces the following result, i.e., $M(\cd)$ is bounded too.
			
			\begin{proposition}\label{21.8.19.1}\sl
				Under the condition (\textbf{H}), the process $M(\cd)$ is bounded, where $(M(\cd),\Phi(\cd),\eta(\cd))$ is the adapted solution of BSDE \rf{2.3}.
			\end{proposition}

			\vspace{-0.5cm}
			
			\section{The main results}\label{MainResult}
			
			In this section, we state our main results, such as the solvability of the corresponding SRE, the closed-loop representation of the open-loop optimal control, and some auxiliary results.
			
			\subsection{Solvability of stochastic Riccati equation}
			As shown in Sun--Xiong--Yong \cite{Sun-Xiong-Yong-21}, the solvability of the corresponding SRE deduces the closed-loop representation of the open-loop optimal control, which is important to deal with Problem (M-SLQ).
			Therefore, the SRE plays a crucial role in studying the SLQ problem, and our core goal is to establish the solvability of the corresponding SRE.
			For this, we introduce the following SRE:
			\begin{equation}\label{21.8.23.6}
				\left\{\begin{aligned}
					\ds d\hat P(s,\a(s))&=-\big[\hat{Q}(s,\a(s))
					+\hat{S}(s, \a(s))^\top \Th(s,\a(s))\big] ds\\
					\ns\ds &\q +\hat \Lambda(s) d W(s) +\hat \zeta(s)\bullet d\wt{N}(s), \q s \in[0, T], \\
					\ns\ds \hat P(T,\a(T))&=G(\a(T)), \q \a(0)=i_0,
				\end{aligned}\right.
			\end{equation}
			where $i_0\in\cS$ is the initial state of $\a(\cd)$, and for any $(s,\o,i)\in[0,T]\ts\Omega\ts\cS$,
			\begin{equation}\label{hatsq}
				\begin{aligned}
					\ds \hat{Q}(s,i)&\deq \hat P(s,i) A(s,i)+A(s,i)^{\top} \hat P(s,i)
					+C(s,i)^{\top} \hat P(s,i) C(s,i) \\
					\ns\ds&\q\ +\hat \Lambda(s) C(s,i)+C(s,i)^{\top} \hat \Lambda(s)+Q(s,i),\\
					\ns\ds \hat{S}(s,i)&\deq B(s,i)^{\top} \hat P(s,i)
					+D(s,i)^{\top} \hat P(s,i) C(s,i)+D(s,i)^{\top}
					\hat \Lambda(s)+S(s,i), \\
					\ns\ds \hat{R}(s,i)&\deq R(s,i)+D(s,i)^{\top} \hat P(s,i) D(s,i),
					\q\
					\Theta(s,i)\deq-\hat{R}(s,i)^{-1}\hat{S}(s,i).
				\end{aligned}
			\end{equation}
			%
			%
			%
			Suppose that the cost functional $J(0,0,i_0; u(\cd))$ is uniformly convex in $u(\cd)$, i.e., for any choice of $i_0\in\cS$, there is a positive constant $\e$ such that
			\begin{align}\label{21.8.21.3.2}
				J(0,0,i_0;u(\cd))\geqslant\e \mathbb{E}\int_{0}^{T}|u(s)|^{2}ds, \q \forall u(\cd) \in \mathcal{U}[0, T].
			\end{align}
			%
			
			%
			\begin{theorem}\label{21.9.3.1} \sl
				Under conditions {\rm (\textbf{H})} and \rf{21.8.21.3.2},
				SRE \rf{21.8.23.6} admits a unique adapted solution $(\hat P(\cd,\a(\cd)), \hat \Lambda(\cd),\hat \zeta(\cd)) \in$ $L_{\mathbb{F}}^{\infty}(0, T ; \mathbb{S}^{n}) \times L_{\mathbb{F}}^{2}(0, T ; \mathbb{S}^{n})\times L_{\dbF}^2(0,T;\mathbb{M}_{D}(\mathbb{S}^n))$. In addition,
				%
				\begin{equation}
					\hat{R}(s,\a(s))=R(s,\a(s))+D(s,\a(s))^\top \hat P(s,\a(s))D(s,\a(s))
					\ges\e I_m,\q \hb{a.e. on $[0,T]$, a.s.}.
				\end{equation}
				%
				%
			\end{theorem}
			The above theorem shows the solvability of SRE \rf{21.8.23.6}, based on which the following closed-loop representation of the open-loop optimal control can be derived.
			\begin{theorem}\label{21.11.26.1} \sl
				Under conditions {\rm (\textbf{H})} and \rf{21.8.21.3.2},
				Problem (M-SLQ) is uniquely open-loop solvable and the unique open-loop optimal control $\left\{u^*(s) \right\}_{s\in[t,T]}$ with respect to the initial triple $(t, \xi,\vf) \in \cD$ has the following linear state feedback:
				\begin{equation}\label{22.6.7.1}
					u^{*}(s)=\Theta(s,\a(s)) X^{*}(s), \q s \in[t, T],
				\end{equation}
				where $\Theta(\cd)$ is defined in \rf{hatsq}
				and $\{X^{*}(s)\}_{s\in[t, T]}$ is the solution of the following closed-loop system:
				\begin{equation*}
					\left\{\begin{aligned}
						\ds d X^{*}(s)=& \big[A(s,\a(s))+B(s,\a(s)) \Theta(s,\a(s))\big] X^{*}(s) d s\\
						\ns\ds & +\big[C(s,\a(s))+D(s,\a(s)) \Theta(s,\a(s))\big] X^{*}(s) d W(s),\q s \in[t, T], \\
						\ns\ds X^{*}(t) =&\ \xi,\q\a(t)=\vf.
					\end{aligned}\right.
				\end{equation*}
			\end{theorem}

			\begin{remark}\rm
				Sun--Xiong--Yong \cite{Sun-Xiong-Yong-21} studied SLQ optimal control problems with random coefficients, proved the solvability of the corresponding SRE, and obtained the closed-loop representation of the open-loop optimal control. Compared with \cite{Sun-Xiong-Yong-21}, on one hand, the above theorems further extend the results in \cite{Sun-Xiong-Yong-21} to the framework within the Markovian regime switching, which is important in continuous-time mean-variance portfolio selection problems (see \autoref{Example}).
				On the other hand, due to the presence of the Markovian regime switching jump in SRE \rf{21.8.23.6}, it is difficult to directly prove the solvability of SRE \rf{21.8.23.6}. To overcome it, we derive some auxiliary results first.
			\end{remark}

			\subsection{Some auxiliary results}\label{Sec3.2}
			
			For simplify the notations, from now on, for any $(s,\o,i)\in[0,T]\ts\Omega\ts\cS$, $x,y,z\in\dbR^n$ and $u\in\dbR^m$, let
			\begin{align*}
				\ds F(s,i,x,y,z,u)&\deq B(s,i)^\top y+D(s,i)^\top z+S(s,i)x+R(s,i)u,\\
				\ns\ds	\widetilde{F}(s,i,x,y,z,u)&\deq A(s,i)^\top y+C(s,i)^\top z+Q(s,i)x+S(s,i)^\top u,\\
				\ns\ds F_0(s,i,x,y, z)\deq\ & F(s,i,x,y,z,0), \q\widetilde{F}_0(s,i,x,y, z)\deq\widetilde{F}(s,i,x,y,z,0).
			\end{align*}
			The following BSDE is called the associated adjoint equation of the state equation \rf{state}:
			\bel{21.8.20.1}\left\{\ba{ll}
			\ds dY(s)=-\widetilde F(s,\a(s),X(s), Y(s), Z(s), u(s))ds+Z(s)dW(s)+\G(s)\bullet d\widetilde{N}(s),\q  s\in[t,T],\\
			\ns\ds Y(T)=G(\a(T))X(T),\q \a(t)=\vf,
			\ea\right.\ee
			where $(X(\cd),u(\cd))$ comes from \rf{state}.
			In order to prove the solvability of SRE \rf{21.8.23.6}, we present an alternative characterization of the solvability of Problem (M-SLQ)$_0$ in terms of the state equation \rf{state} and the adjoint equation \rf{21.8.20.1}, and then show that Problems (M-SLQ) and (M-SLQ)$_0$ are equivalent.
			\begin{theorem}\label{Theorem4.1} \sl
				Let {\rm (\textbf{H})} hold and let the initial triple $(t,\xi,\vf)\in\cD$ be given.
				A process $u^*(\cd)\in\cU[t,T]$ is an open-loop optimal control of Problem (M-SLQ)$_0$ with respect to $(t,\xi,\vf)$ if and only if the following two conditions hold:
				\begin{itemize}
					\item [(i)] The mapping $u(\cd)\mapsto J_0(t,0,\vf;u(\cd))$ is convex, or equivalently,
					\bel{21.9.15.2}J_0(t,0,\vf;u(\cd))\ges0,\q \forall u(\cd)\in\cU[t,T].\ee
					\item [(ii)] The following stationarity condition hold:
					\begin{align}\label{4.2}
						F(s,\a(s),X^*(s),  Y^*(s), Z^*(s), u^*(s))=
						0,\q \hb{a.e. $s\in[t,T]$, a.s.,}
					\end{align}
					where $X^*(\cd)$ is the solution of SDE \rf{state} and $(Y^*(\cd),Z^*(\cd),\G^*(\cd))$ is the solution of BSDE \rf{21.8.20.1} with $u(\cd)$ replaced by $u^*(\cd)$.
				\end{itemize}
			\end{theorem}
			\begin{remark} \rm
				In the assertion (ii) of \autoref{Theorem4.1}, the quadruple $(X^*(\cd),Y^*(\cd),Z^*(\cd),\G^*(\cd))$ is essentially the solution of the following decoupled FBSDE:
				\bel{4.1}\hspace{-0.5cm}\left\{\ba{ll}
				\ds dX^*(s)=\big[A(s,\a(s))X^*(s)+B(s,\a(s))u^*(s)\big]ds
				+\big[C(s,\a(s))X^*(s)+D(s,\a(s))u^*(s)\big]dW(s),\\
				\ns\ds dY^*(s)=-\widetilde{F}(s,\a(s),X^*(s), Y^*(s), Z^*(s), u^*(s))ds+Z(s)dW(s)+\G^*(s)\bullet d\widetilde{N}(s),\q s\in[t,T],\\
				\ns\ds X^*(t)=\xi,\q\a(t)=\vf,\q Y^*(T)=G(\a(T))X^*(T).
				\ea\right.\ee
				%
				%
			\end{remark}

			\begin{proposition}\label{Theorem4.2} \sl
				Let the condition {\rm (\textbf{H})} hold. For any given initial triple $(t,\xi,\vf)\in\cD$, a control $u^*(\cd)\in\cU[t,T]$ is optimal for Problem (M-SLQ)$_0$ if and only if it is optimal for Problem  ${(\hb{M-SLQ})}$.
			\end{proposition}
			%
			%
			Based on \autoref{Theorem4.1} and \autoref{Theorem4.2}, we have the following results, 
			which are useful to prove \autoref{21.9.3.1} later. For this, denote by $\tau$ an $\mathbb{F}$-stopping time with values in $[0, T]$ and denote by $\mathcal{T}[a,b]$ the set of all $\mathbb{F}$-stopping times valued in the interval $[a,b]$ with $a,b\in[0,T]$.
			\begin{corollary}\label{21.9.3.3} \sl
				Let {\rm (\textbf{H})} hold and suppose that $u^*(\cd)\in\cU[t,T]$  is an open-loop optimal control w.r.t. the initial triple $(t,\xi,\vf)\in\cD$, then
				$${V}(t,\xi,\vf)={J}(t,\xi,\vf;u^*(\cd))=\langle Y^*(t),\xi\rangle,$$
				where the quadruple $(X^*(\cd),Y^*(\cd),Z^*(\cd),\G^*(\cd))$ is the adapted solution of FBSDE \rf{4.1} w.r.t. $u^*(\cd)$.
			\end{corollary}
			%
			%
			
			\begin{corollary}\label{Corollary4.5}\sl
				Let {\rm (\textbf{H})} hold and suppose that $u^*(\cd)\in\cU[t,T]$  is an open-loop optimal control w.r.t. the initial triple $(t,\xi,\vf)\in\cD$, then for any stopping time $\t\in\cT[t,T]$, the restriction
				$$u^*(\cd)|_{[\t,T]}\deq \{ u^*(s); s\in[\t,T] \}$$
				over the time horizon $[\t,T]$ remains optimal w.r.t. the initial triple $(\t,X^*(\t),\a(\t))$, where $X^*(\cd)$ is the solution of the forward equation of \rf{4.1}.
				
			\end{corollary}
			
			\begin{proposition}\label{21.8.21.4} \sl
				Let {\rm (\textbf{H})} and \rf{21.8.21.3.2} hold, then for any $\t\in\cT[0,T)$,
				$$J_0(\tau,0,\a(\t); u(\cd)) \geqslant \e \mathbb{E} \int_{\tau}^{T}|u(s)|^{2} d s, \q \forall u(\cd) \in \mathcal{U}[\tau, T] .$$
				As a direct consequence, Problem (M-SLQ) is uniquely solvable.
			\end{proposition}
			Next, we state some properties for the stochastic value flow ${V}(t,\xi,\vf)$, which are important to prove \autoref{21.9.3.1} later.
			Recall that $\{e_{j}\}_{j=1}^{n}$ is the standard basis of $\mathbb{R}^{n}$.
			\begin{proposition}\label{21.8.21.2} \sl
				Let {\rm (\textbf{H})} hold. Suppose that Problem (M-SLQ) is solvable at the initial triple $(t,e_j,\vf)$ with $1\les j\les n$. Let the state-control pair $\{\left(X_{j}(s), u_{j}(s)\right) \}_{s\in[t,T]}$ of SDE \rf{state} be an open-loop optimal pair w.r.t. $(t,e_j,\vf)$, and
				let $\{\left(Y_{j}(s), Z_{j}(s),\G^j(s)\right)\}_{s\in[t,T]}$ be the solution of the associated adjoint equation \rf{21.8.20.1}. Denote $\boldsymbol{\G}_{kl}(s)\deq (\G_{kl}^1(s),\cdots, \G_{kl}^n(s))$ and
				\begin{equation}\label{22.6.10.1}
					\begin{aligned}
						\ds	\BX(s)&\deq \left(X_{1}(s), \cdots, X_{n}(s)\right), \q
						\BU(s)\deq \left(u_{1}(s), \cdots, u_{n}(s)\right),\q  s\in[t,T],\\
						\ns\ds 	\BY(s)&\deq \left(Y_{1}(s), \cdots, Y_{n}(s)\right),\q
						\BZ(s)\deq \left(Z_{1}(s), \cdots, Z_{n}(s)\right),\q \boldsymbol{\G}(s)\deq \big(\boldsymbol{\G}_{kl}(s)\big)_{{k,l}\in\cS}.
					\end{aligned}
				\end{equation}
				Then the quintuple of matrix-valued processes $\big(\boldsymbol{X}(\cd), \boldsymbol{U}(\cd), \boldsymbol{Y}(\cd),\boldsymbol{Z}(\cd),\boldsymbol{\G}(\cd)\big)$ satisfies the following FBSDE:
				\begin{equation}\label{22.6.25.2}
					\left\{\ba{ll}
					\ds d\BX(s)=\big[A(s,\a(s))\BX(s)+B(s,\a(s))\BU(s)\big]ds
					+\big[C(s,\a(s))\BX(s)+D(s,\a(s))\BU(s)\big]dW(s),\\
					\ns\ds d\BY(s)=-\widetilde{F}(s,\a(s),\BX(s), \BY(s), \BZ(s), \BU(s))ds 
					+\BZ(s)dW(s)+\boldsymbol{\G}(s)\bullet d\widetilde{N}(s),\q s\in[t,T],\\
					\ns\ds \BX(t)=I_n,\q \a(t)=\vf,\q\BY(T)=G(\a(T))\BX(T),
					\ea\right.
				\end{equation}
				and
				\begin{align}\label{21.8.20.3}
					F(s,\a(s),\BX(s), \BY(s), \BZ(s), \BU(s))=0, \q \hb{a.e. $s\in[t,T]$, a.s.}.
				\end{align}
				In addition, for every $\xi \in L_{\sF_{t}}^{\infty}\left(\Omega ; \mathbb{R}^{n}\right)$, the state-control pair
				$(\BX \xi, \BU \xi)=\{(\BX(s) \xi, \BU(s) \xi)\}_{s\in[t,T]}$
				is optimal with respect to $(t,\xi,\vf)$ and the triple
				$(\BY \xi, \BZ \xi, \boldsymbol{\G}\circ\xi )=\{(\BY(s) \xi, \BZ(s) \xi,\boldsymbol{\G}(s)\circ\xi) \}_{s\in[t,T]}$
				with $\mathbf{\G}(s)\circ\xi\deq(\G_{kl}(s)\xi)_{k,l\in \cS}$
				solves the adjoint BSDE \rf{21.8.20.1} associated with the state-control pair $(\BX \xi, \BU \xi)$.
			\end{proposition}
			Based on the above results, we have the following theorems, which present a quadratic form of the stochastic value flow into a bounded and left-continuous process.
			\begin{theorem}\label{21.8.21.5} \sl
				Let (\textbf{H}) hold. For any given $t\in[0,T]$ and
				$\vf\in L_{\sF_{t}^\a}^{2}(\Omega ; \cS)$,
				if Problem (M-SLQ) is solvable at the initial time $t$,
				then there is a process $P:[0,T]\ts\cS\ts\Omega\rightarrow\dbS^n$, which is $\dbF$-adapted, such that
				\begin{equation}\label{21.8.21.6}
					{V}(t,\xi,\vf)=\langle P(t,\vf) \xi, \xi\rangle, \q \forall \xi \in L_{\sF_{t}}^{\infty}(\Omega ; \mathbb{R}^{n}).
				\end{equation}
			\end{theorem}
			\begin{theorem}\label{21.8.22.4} \sl
				Let conditions {\rm (\textbf{H})} and \rf{21.8.21.3.2} hold.
				Then the process $P=\{P(t,i); (t,i)\in[0,T]\ts \cS\}$
				appeared in \rf{21.8.21.6} is bounded and left-continuous.
			\end{theorem}
			
			%
			%
			%
			Finally, based on \autoref{21.8.21.5} and \autoref{21.8.22.4}, we introduce a stopped stochastic linear-quadratic problem and present some results for it, which is useful to the proof of \autoref{21.9.3.1} too.  Recall that $\tau\in\cT(0,T]$ is an $\mathbb{F}$-stopping time, and set
			$$\mathcal{D}^{\tau}\deq\big\{(\sigma,\xi,\vartheta)\ |\ \sigma \in \mathcal{T}[0, \tau),\ \xi\in L^2_{\sF_{\sigma}}(\Omega;\dbR^n),\ \vartheta\in L_{\sF_\sigma^\alpha}^2(\Omega;\mathcal{S})\big\}.$$	

			\noindent\bf Problem (M-SLQ)$^\t$:  \rm
			For any given initial triple $(\sigma, \xi,\vartheta) \in \mathcal{D}^{\tau}$,
			find a control $u^{*}(\cd) \in \mathcal{U}[\sigma, \tau]$ such that
			$$J^\t(\sigma,\xi,\vartheta;u^{*}(\cdot))
			=\essinf_{u(\cd)\in\cU[\sigma, \tau]}J^\t(\sigma,\xi,\vartheta;u(\cdot))
			\deq V^\t(\sigma,\xi,\vartheta),$$
			where the cost functional
			{\small$$\label{eq:Jtau}
				\ds {J}^\t(\sigma,\xi,\vf;u(\cd))\deq\dbE_{\sigma}\left[\big\lan P(\tau,\a(\t)) X(\tau), X(\tau)\big\ran
				+\int_{\sigma}^{\tau}\llan\begin{pmatrix}Q(s,\a(s))&S(s,\a(s))^\top\\S(s,\a(s))&R(s,\a(s))\end{pmatrix}
				\begin{pmatrix}X(s)\\ u(s)\end{pmatrix},
				\begin{pmatrix}X(s)\\u(s)\end{pmatrix}\rran ds\right],
				$$}
			and $X(\cd)$ is the solution of \rf{state} w.r.t. the initial triple $(\sigma, \xi,\vartheta)$ over the stochastic interval $[\sigma, \tau]$.
			
			\ms
			
			Similar to \autoref{prob-expectation}, when the above conditional expectation $\dbE_\sigma[\cdot]$  degenerates to the expectation $\dbE$, we denote the related problem, cost functional, and value function by {\bf Problem (M-SLQ)$^\t_0$}, $J^\t_0(t,\xi,\vf;u(\cd))$, and $V^\t_0(t,\xi,\vf)$, respectively. 			

				\begin{proposition}\label{21.8.31.3} \sl
					Let conditions {\rm (\textbf{H})} and \rf{21.8.21.3.2} hold. Then
					\begin{itemize}
						\item [{\rm (i)}] Problem (M-SLQ)$^\t$ is uniquely solvable at every $\sigma \in \mathcal{T}[0, \tau)$.
						\item [{\rm (ii)}] If $u^{*}(\cd) \in \mathcal{U}[\sigma, T]$ is an open-loop optimal control of Problem (M-SLQ) w.r.t. $(\sigma, \xi,\vf) \in \mathcal{T}[0, \tau) \times L_{\sF_{\sigma}}^{\infty}\left(\Omega ; \mathbb{R}^{n}\right)\ts L_{\sF_{\sigma}^\a}^{2}\left(\Omega ; \cS\right)$,
						then the restriction $\left.u^{*}(\cd)\right|_{[\sigma, \tau]}$ of $u^{*}(\cd)$ on the interval $[\sigma, \tau]$ is also an open-loop optimal control of Problem (M-SLQ)$^\t$ w.r.t. the same initial triple $(\sigma, \xi,\vf)$.
						\item [{\rm (iii)}] The stochastic value flow $V^\t(\cd)$ of Problem (M-SLQ)$^\t$ has the following form:
						$$V^\t(\sigma, \xi,\vf)=\langle P(\sigma,\vf) \xi, \xi\rangle, \quad
						\forall(\sigma, \xi,\vf) \in \mathcal{T}[0, \tau) \times L_{\sF_{\sigma}}^{\infty}\left(\Omega ; \mathbb{R}^{n}\right)\ts L_{\sF_{\sigma}^\a}^{2}\left(\Omega ; \cS\right),$$%
						where $P(\cd,\a(\cd))$ is the process appeared in \rf{21.8.21.6}.
					\end{itemize}
				\end{proposition}

				\section{Proof of the solvability of SRE}\label{ProofSRE}
				
				In this section, based on the auxiliary results of \autoref{Sec3.2}, we prove the solvability of SRE \rf{21.8.23.6}.
				First, we prove that SRE \rf{21.8.23.6} is uniquely solvable, and then prove that the first component $\hat P(\cd,\a(\cd))$ of SRE \rf{21.8.23.6} is exactly the process $P(\cd,\a(\cd))$ appeared in \rf{21.8.21.6}. Finally, as a by-product, the open-loop optimal control is represented as linear feedback of the state.

				\ms
				
				We will use several lemmas to prove \autoref{21.9.3.1}. Recall that SDE \rf{state} is the state equation, BSDE \rf{21.8.20.1} is the associated adjoint equation, and $\{e_{j}\}_{j=1}^{n}$ is the standard basis of $\mathbb{R}^{n}$.

				\begin{lemma}\label{21.9.2.1} \sl
					Suppose the conditions {\rm (\textbf{H})} and \rf{21.8.21.3.2} hold.  Denote by $\{X_{j}(s)\}_{s\in [0,T]}$  the (unique) open-loop optimal state process corresponding to the initial triple $(t,\xi,\vf)=(0, e_{j},i_0)$.
					Then the $\mathbb{R}^{n \times n}$-valued process $\{\boldsymbol{X}(s) \}_{s\in[0,T]}$ with $\boldsymbol{X}(s)\deq\left(X_{1}(s), \ldots, X_{n}(s)\right)$ is invertible.
				\end{lemma}

				\begin{proof}
					Let $u_{j}(\cd) \in \mathcal{U}[0, T]$ be the unique open-loop optimal control w.r.t.  $(0, e_{j},i_0)$ so that
					\begin{equation*}
						\hspace{-0.38cm}\left\{\begin{aligned}
							\ds &dX_j(s)=\big[A(s,\a(s))X_j(s)+B(s,\a(s))u_j(s)\big]ds \\
							\ns\ds &\qq\qq +\big[C(s,\a(s))X_j(s)+D(s,\a(s))u_j(s)\big]dW(s), \q s\in[0,T],\\
							\ns\ds & \ds X(0)=e_j,\q\a(0)=i_0.
						\end{aligned}\right.
					\end{equation*}
					Then, with $\boldsymbol{U}(s)=\left(u_{1}(s), \ldots, u_{n}(s)\right)$, one has
					\bel{21.9.4.1}\hspace{-0.38cm}\left\{\begin{aligned}
						\ds &d\BX(s)=\big[A(s,\a(s))\BX(s)+B(s,\a(s))\BU(s)\big]ds\\
						\ns\ds &\qq\q\ \ +\big[C(s,\a(s))\BX(s)+D(s,\a(s))\BU(s)\big]dW(s),\q s\in[0,T],\\
						\ns\ds & \BX(0)=I_n,\q\a(0)=i_0.
					\end{aligned}\right.\ee
					We define the following stopping time (at which, for the first time, $\BX(\cd)$ is not invertible)
					$$\theta(\omega)=\inf \big\{s \in[0, T] ; \operatorname{det}(\boldsymbol{X}(s, \omega))=0\big\},$$
					where we use the convention that the infimum of the empty set is infinity.
					It should be pointed out that if we prove that  $\mathbb{P}(\theta=\infty)=1$, i.e., the set
					$\mathbb{O}\deq \{\omega \in \Omega : \theta(\omega) \leqslant T\}$
					has probability zero, then we can get that $X(\cd)$ is invertible.
					
					\ms
					
					Suppose the contrary and set $\tau=\theta \wedge T$. Then $\tau$ is in $\cT(0,T]$ too. When $\tau=\theta$ on $\mathbb{O}$, we know that $\boldsymbol{X}(\tau)$ is not invertible on $\mathbb{O}$ by the definition of $\theta$. Thus, we can choose an $\sF_{\tau}$-measurable, $\mathbb{S}^{n}$-valued, positive semi-definite random matrix $H$ with $|H|=1$ on $\mathbb{O}$ such that
					$H(\omega) \boldsymbol{X}(\tau(\omega), \omega)=0$ for any $\omega \in \Omega$.
					
					\ms
					
					Note that $P=\{P(t,i); (t,i)\in[0,T]\ts \cS\}$
					is bounded, left-continuous, and satisfies \rf{21.8.21.6} according to \autoref{21.8.21.5} and \autoref{21.8.22.4}.
					For $\sigma\in\cT[0,\t]$ with $s\in[\sigma,\t]$, we consider the following equation
					\bel{21.8.22.1}\left\{\ba{ll}
					\ds dX(s)=\big[A(s,\a(s))X(s)+B(s,\a(s))u(s)\big]ds
					+\big[C(s,\a(s))X(s)+D(s,\a(s))u(s)\big]dW(s),\\
					\ns\ds X(\sigma) =\xi,\q \a(\sigma)=\vf,
					\ea\right.\ee
					and the following auxiliary cost functional
					\begin{equation}\label{21.8.31.2}
						J^H(\sigma, \xi,\vf; u(\cd)) \triangleq J^\t(\sigma, \xi,\vf; u(\cd))
						+\mathbb{E}_{\sigma}\big[\langle H X(\tau), X(\tau)\rangle\big].
					\end{equation}
					Consider the problem of minimizing the cost functional \rf{21.8.31.2} subject to the state equation \rf{21.8.22.1}, called Problem ${(\hb{M-SLQ})_H}$. Moreover, denote by ${V_H}(\cdot)$ the associated stochastic value flow. Then the following two assertions hold:
					\begin{itemize}
						\item[(i)] For any choice of $\sigma \in \mathcal{T}[0, \tau)$ and $\vartheta\in L_{\sF_t^\alpha}^2(\Omega;\mathcal{S})$ with $\a(\sigma)=\vf$, the following inequality holds:
						\begin{align*}
							\dbE\big[{J^H}(\sigma, 0,\vf; u(\cd))\big] \geqslant \dbE\big[J^\t(\sigma, 0,\vf ; u(\cd)) \big]\geqslant \e \mathbb{E}\big[\int_{\sigma}^{\tau}|u(s)|^{2} d s\big], \q  \forall u(\cd) \in \mathcal{U}[\sigma, \tau] .
						\end{align*}
					\end{itemize}
					\begin{itemize}
						\item[(ii)] The restriction $u_{j}^{\tau}(\cd)=\left.u_{j}(\cd)\right|_{[0, \tau]}$  over the time inverval $[0, \tau]$ is optimal for both Problem (M-SLQ)$^\t$ and  Problem ${(\hb{M-SLQ})_H}$ w.r.t. the same initial triple $(0, e_{j},i_0)$.
					\end{itemize}
					For the above assertion (i), in fact, the first inequality is true since $H$ is positive semi-definite, and the second inequality comes from the proof of the assertion (i) of  \autoref{21.8.31.3}. Then, as a result of \autoref{21.8.21.4}, both Problems (M-SLQ)$^\t$ and ${(\hb{M-SLQ})_H}$ are uniquely solvable at any $\sigma \in \mathcal{T}[0, \tau).$

					\ms
					
					For the above assertion (ii), on one hand, by the assertion (ii) of \autoref{21.8.31.3}, we see that $u_{j}^{\tau}(\cd)$ is optimal for Problem (M-SLQ)$^\t$ w.r.t. the initial triple $(0, e_{j},i_0)$.
					On the other hand, according to \autoref{Theorem4.1}, in order to show that $u_{j}^{\tau}(\cd)$ is an open-loop optimal control for Problem ${(\hb{M-SLQ})_H}$ w.r.t. $\left(0, e_{j},i_0\right)$,  it suffices to prove that the solution $(X_{j}^{\tau}(\cd), Y_{j}^{\tau}(\cd), Z_{j}^{\tau}(\cd),\G^{j,\t}(\cd))$ of the following FBSDE:
					\bel{8.31.4}\left\{\ba{ll}
					\ds dX_{j}^{\tau}(s)=\big[A(s,\a(s))X_{j}^{\tau}(s)+B(s,\a(s))u_{j}^{\tau}(s)\big]ds
					+\big[C(s,\a(s))X_{j}^{\tau}(s)+D(s,\a(s))u_{j}^{\tau}(s)\big]dW(s),\\
					\ns\ds dY_{j}^{\tau}(s)=-\widetilde{F}(s,\a(s), X_j^{\tau}(s),  Y_{j}^{\tau}(s), Z_{j}^{\tau}(s), u_{j}^{\tau}(s))ds
					%
					+Z_{j}^{\tau}(s)dW(s)+\G^{j,\tau}(s)\bullet d\widetilde{N}(s),\ s\in[0,\t],\\
					\ns\ds 	X_{j}^{\tau}(0)=e_j,\q\a(0)=i_0,\q  Y_{j}^{\tau}(\tau)=[P(\tau,\a(\t))+H] X_{j}^{\tau}(\tau),
					\ea\right.\ee
					satisfying the following stationarity condition:
					\begin{align}\label{21.8.31.5}
						F(s,\a(s), X_j^{\tau}(s),  Y_{j}^{\tau}(s), Z_{j}^{\tau}(s),
						u_{j}^{\tau}(s))=0,\q \hb{a.e. $s\in[0,\t]$, a.s.}.
					\end{align}
					Note that $X_{j}^{\tau}(s)=\boldsymbol{X}(s) e_{j}$ for $0 \leqslant s \leqslant \tau$. Thus, by the choice of $H$, we have
					\begin{equation}\label{21.8.31.6}
						H X_{j}^{\tau}(\tau)=H \boldsymbol{X}(\tau) e_{j}=0.
					\end{equation}
					It follows that the terminal value $Y_{j}^{\tau}(\tau)=P(\tau,\a(\t)) X_{j}^{\tau}(\tau)$, which implies that FBSDE \rf{8.31.4}
					is exactly the FBSDE associated with Problem (M-SLQ)$^\t$.
					Then from \autoref{Theorem4.1} again, the stationarity condition \rf{21.8.31.5} follows from the fact that $u_{j}^{\tau}(\cd)$ is an open-loop optimal control of Problem (M-SLQ)$^\t$ with respect to $\left(0, e_{j},i_0\right)$.
					
					\ms
					
					Now, for Problem ${(\hb{M-SLQ})_H}$, by the above assertion (i) and \autoref{21.8.21.5}, there is a bounded, left-continuous and $\mathbb{S}^{n}$-valued process $P_H(\cd,i)$ such that
					\begin{equation}\label{22.6.25.1}
						{V_H}(\sigma,\xi,\vf)=\langle P_H(\sigma,\vf)\xi,\xi\rangle,\q \forall (\sigma,\xi,\vf)\in\cT[0,\t)\ts L_{\sF_{\sigma}}^{\infty}(\Omega ; \mathbb{R}^{n})\ts
						L_{\sF_{\sigma}^\a}^{2}(\Omega ; \cS).
					\end{equation}
					Next, we prove that the process $P(\cd,\a(\cd))$ appeared in \rf{21.8.21.6} equals to $P_H(\cd,\a(\cd))$ appeared in \rf{22.6.25.1}, i.e.,
					\begin{equation}\label{21.8.31.7}
						P(t,\a(t))={P}_H(t,\a(t)),  \q  t\in[0,\t).
					\end{equation}
					
					By the above assertion (ii), we see that $\{\left(X_{j}^{\tau}(s), u_{j}^\t(s)\right)\}_{s\in [0,\tau]}$ is the optimal state-control pair for both Problem (M-SLQ)$^\t$ and Problem ${(\hb{M-SLQ})_H}$ w.r.t. the initial triple $\left(0, e_{j},i_0\right)$.
					Let
					$$
					\boldsymbol{X}^{\tau}(s) \triangleq\left(X_{1}^{\tau}(s), \ldots, X_{n}^{\tau}(s)\right), \q
					\boldsymbol{U}^{\tau}(s) \triangleq\left(u_{1}^{\tau}(s), \ldots, u_{n}^{\tau}(s)\right),\q s\in[0,\tau],$$
					and take an arbitrary $x \in \mathbb{R}^{n}$. From \autoref{21.8.21.2}, we have that the pair $\left(\boldsymbol{X}^{\tau} x, \boldsymbol{U}^{\tau} x\right)$ is the optimal state-control for both Problem (M-SLQ)$^\t$ and Problem ${(\hb{M-SLQ})_H}$ w.r.t. the initial triple $(0,x,i_0)$. Moreover, by \autoref{Corollary4.5}, the following pair
					$$\big(\boldsymbol{X}^{\tau}(s) x, \boldsymbol{U}^{\tau}(s) x\big), \q s\in[t,\tau],$$
					remains optimal w.r.t. $\left(t, \boldsymbol{X}^{\tau}(t) x,\a(t)\right)$ for every $t\in[0,\tau)$.
					Thus, note that $H \boldsymbol{X}^{\tau}(\tau)=0$ by \rf{21.8.31.6},
					$$
					\begin{aligned}
						\ds {V_H}\left(t, \boldsymbol{X}^{\tau}(t) x,\a(t)\right)
						&={J_H}\left(t, \boldsymbol{X}^{\tau}(t) x,\a(t) ; \boldsymbol{U}^{\tau}(t) x\right)
						=J\left(t, \boldsymbol{X}^{\tau}(t) x,\a(t) ; \boldsymbol{U}^{\tau}(t) x\right)
						+\mathbb{E}\left\langle H \boldsymbol{X}^{\tau}(\tau) x, \boldsymbol{X}^{\tau}(\tau) x\right\rangle \\
						\ns\ds &=J\left(t, \boldsymbol{X}^{\tau}(t) x,\a(t) ; \boldsymbol{U}^{\tau}(t) x\right)
						={V}\left(t, \boldsymbol{X}^{\tau}(t) x,\a(t)\right).
					\end{aligned}$$
					Noting that $\boldsymbol{X}^{\tau}(t)=\boldsymbol{X}(t)$ when $t\in[0,\tau)$, we deduce from the above that
					$$\langle{P}_H(t,\a(t)) \boldsymbol{X}(t) x, \boldsymbol{X}(t) x\rangle
					={V_H}\left(t, \boldsymbol{X}^{\tau}(t) x,\a(t)\right)={V}\left(t, \boldsymbol{X}^{\tau}(t) x,\a(t)\right)=\langle P(t,\a(t))
					\boldsymbol{X}(t) x, \boldsymbol{X}(t) x\rangle .$$
					Since $x \in \mathbb{R}^{n}$ is arbitrary, it follows that
					$$\boldsymbol{X}(t)^{\top} P(t,\a(t)) \boldsymbol{X}(t)=\boldsymbol{X}(t)^{\top} {P}_H(t,\a(t)) \boldsymbol{X}(t), \q t\in[0,\t).$$
					%
					From the definition of $\tau,$ we see that $\boldsymbol{X}$ is invertible on $[0, \tau)$, which implies that
					the relation \rf{21.8.31.7} holds.
					Thus, for some choice of $i\in\cS$, one has
					\begin{equation}\label{22.7.6.1}
						P(t,i)=P_H(t,i),\q t\in[0,\t).
					\end{equation}
					However, on the other hand, ${P}_H(\tau,i)=P(\tau,i)+H$, and both $P(\cd,i)$ and ${P}_H(\cd,i)$ are left-continuous.
					Finally, in \rf{22.7.6.1}, letting $t \uparrow \tau$ yields a contradiction: $P(\tau,i)=P(\tau,i)+H$, since $|H|=1$ on $\mathbb{O}$.
				\end{proof}
				
				Note that under the conditions of \autoref{21.9.3.1}, Problem (M-SLQ) is uniquely solvable by \autoref{21.8.21.4}.
				Let $\{\left(X_{j}(s), u_{j}(s)\right)\}_{s\in[0,T]}$ be the open-loop optimal pair and
				$\left\{\left(Y_{j}(s), Z_{j}(s),\G^j(s)\right) \right\}_{s\in[t,T]}$
				be the adapted solution of the adjoint BSDE corresponding to  the initial triple $(0,e_j,i_0)$ for each $j=1,2,...,n$, respectively.
				Now, with $\boldsymbol{\G}_{kl}(s)\deq (\G_{kl}^1(s),\cdots, \G_{kl}^n(s)),$
				\autoref{21.8.21.2} implies that the matrix-valued processes $(\boldsymbol{X}(\cd), \boldsymbol{U}(\cd), \boldsymbol{Y}(\cd),\boldsymbol{Z}(\cd),\boldsymbol{\G}(\cd))$ defined by \rf{22.6.10.1}  satisfy FBSDE \rf{22.6.25.2} with the initial triple $(0,I_n,i_0)$ over  $[0,T]$.
				Moreover, the following stationary holds:
				\begin{align}\label{21.9.3.5}
					F(s,\a(s), \BX(s), \BY(s), \BZ(s), \BU(s))=0, \q \hb{a.e. $s\in[0,T]$, a.s.}.
				\end{align}
				In addition, on one hand, from \autoref{21.9.2.1} we see that $\BX(\cd)$ is invertible. On the other hand, from \autoref{21.8.21.5} and \autoref{21.8.22.4}, there is a process
				$P:[0,T]\ts\cS\ts\Omega\rightarrow\dbS^n$, which is left-continuous, bounded and $\dbF^W$-adapted, that satisfies \rf{21.8.21.6}.

				\begin{lemma} \sl
					Suppose that condition {\rm (\textbf{H})} holds,
					and there is a constant $\e>0$ such that \rf{21.8.21.3.2} holds. Then
					\begin{equation}\label{21.9.6.1}
						P(s,\a(s))=\boldsymbol{Y}(s) \boldsymbol{X}(s)^{-1},\q s\in[0,T],
					\end{equation}
					where $P(\cd,\a(\cd))$ is the process appeared in \rf{21.8.21.6} and the pair $(\boldsymbol{X}(\cd),\boldsymbol{Y}(\cd))$ is defined in \rf{22.6.10.1}.
				\end{lemma}

				\begin{proof}
					For arbitrary $x \in \mathbb{R}^{n}$ and $s\in[0,T]$, set that
					$$\begin{aligned}
						X^{*}(s)=\boldsymbol{X}(s) x,\q u^{*}(s)=\boldsymbol{U}(s) x,\q
						Y^{*}(s)=\boldsymbol{Y}(s) x,\q Z^{*}(s)=\boldsymbol{Z}(s) x,\q
						\G^{*}(s)=\boldsymbol{\G}(s)\circ x.
					\end{aligned}$$
					Then, on  one hand, \autoref{21.8.21.2} implies that the pair $\left(X^{*}(\cd), u^{*}(\cd)\right)$ is an open-loop optimal pair w.r.t. the initial triple $(0, x,i_0)$, and that the triple $\left(Y^{*}(\cd), Z^{*}(\cd), \Gamma^*(\cd)\right)$ is the adapted solution to the adjoint BSDE associated with $\left(X^{*}(\cd), u^{*}(\cd)\right)$.
					On the other hand, for every $s \in[0, T]$, \autoref{Corollary4.5} implies that the restriction $(X^{*}(\cd)|_{[s, T]},u^{*}(\cd)|_{[s, T]})$ of $(X^{*}(\cd), u^{*}(\cd))$ over the interval $[s, T]$ remains optimal with respect to $(s, X^{*}(s),\a(s))$. Thus, from  \autoref{21.9.3.3}, we have that
					$$
					{V}\left(s, X^{*}(s),\a(s)\right)=\left\langle Y^{*}(s), X^{*}(s)\right\rangle.
					$$
					Owing to the relation \rf{21.8.21.6}, the above equation implies that
					$$
					\begin{aligned}
						\ds	x^{\top} \boldsymbol{X}(s)^{\top} P(s,\a(s)) \boldsymbol{X}(s) x &=\langle P(s,\a(s)) \boldsymbol{X}(s) x, \boldsymbol{X}(s) x\rangle=\left\langle P(s,\a(s)) X^{*}(s), X^{*}(s)\right\rangle
						={V}\left(s, X^{*}(s),\a(s)\right)\\
						\ns\ds 	&=\left\langle Y^{*}(s), X^{*}(s)\right\rangle
						=\langle\boldsymbol{Y}(s) x, \boldsymbol{X}(s) x\rangle=x^{\top} \boldsymbol{X}(s)^{\top} \boldsymbol{Y}(s) x.
					\end{aligned}
					$$
					Since $x \in \mathbb{R}^{n}$ is arbitrary, we deduce that
					$$\boldsymbol{X}(s)^{\top} P(s,\a(s)) \boldsymbol{X}(s)=\boldsymbol{X}(s)^{\top} \boldsymbol{Y}(s).$$
					The desired result follows from the fact that $\boldsymbol{X}(\cd)$ is invertible.
				\end{proof}

				In the following, please keep in mind that $P(\cd,\a(\cd))$ represents the process appeared in \rf{21.8.21.6} and $(\boldsymbol{X}(\cd), \boldsymbol{U}(\cd), \boldsymbol{Y}(\cd),\boldsymbol{Z}(\cd),\boldsymbol{\G}(\cd))$ defined by \rf{22.6.10.1}  satisfy FBSDE \rf{22.6.25.2} with $(0,I_n,i_0)$ over  $[0,T]$.

				\begin{lemma}\label{21.9.8.2} \sl
					Suppose  {\rm (\textbf{H})} and \rf{21.8.21.3.2} hold. Then, with relation \rf{21.9.6.1} and notation,
					\begin{equation}\label{21.9.6.2}
						\begin{aligned}
							\ds &\     \Theta(s,\a(s))= \BU(s) \boldsymbol{X}(s)^{-1}, \quad\Lambda(s)= \Pi(s,\a(s))-P(s,\a(s))[C(s)+D(s) \Theta(s,\a(s))], \\
							\ns\ds &\ \Pi(s,\a(s))=\BZ(s) \boldsymbol{X}(s)^{-1}, \q
							\zeta(s)=(\z_{k,l}(s))_{k,l\in\cS} \mbox{  \rm with  }
							\z_{kl}(s)\deq \boldsymbol{\G}_{kl}(s)\boldsymbol{X}^{-1}(s), \q s\in[0,T], \end{aligned} 
				\end{equation}
				we have that the triple $(P(\cd,\a(\cd)),\L(\cd),\zeta(\cd))$ satisfies the following BSDE:
				\begin{equation}\label{21.9.3.4}
					\left\{\begin{aligned}
						\ds	dP(s,\a(s))=&-\big[\hat{Q}(s,\a(s))
						+\hat{S}(s, \a(s))^\top \Th(s,\a(s))\big] ds+\Lambda(s) d W(s)
						+\z(s)\bullet d\wt{N}(s), \q s \in[0, T], \\
						\ns\ds 	P(T,\a(T))=&G(\a(T)), \q \a(0)=i_0,
					\end{aligned}\right.
				\end{equation}
				where $\hat Q(\cd)$ and $\hat S(\cd)$ are defined in \rf{hatsq}.
				Moreover, $\Lambda(\cd)$ is symmetric and the following relation holds:
				\begin{align}\label{21.9.7.1}
						\hat{S}(s,\a(s))+\hat{R}(s,\a(s)) \Theta(s,\a(s))=0, \q \text { a.e. }s\in[0, T], \text { a.s.} .
				\end{align}
				%
			\end{lemma}

			\begin{proof}
				First, by the relation (\ref{21.8.21.6}), we have
				$$
				\langle G(\a(T)) \xi, \xi\rangle={V}(T, \xi,\a(T))=\langle P(T,\a(T)) \xi, \xi\rangle, \q  \forall \xi \in L_{\sF_{T}}^{\infty}\left(\Omega ; \mathbb{R}^{n}\right),
				$$
				which leads to that $P(T,\a(T))=G(\a(T))$. Note that $\boldsymbol{X}(\cd)$ satisfies SDE \rf{21.9.4.1} and is invertible, so its invertibility (denoted by $\boldsymbol{X}^{-1}(\cd)$) exists and satisfies the following SDE:
				$$
				d \boldsymbol{X}(s)^{-1}=\Xi(s,\a(s)) d s+\Delta(s,\a(s)) d W(s), \q  s \in[0, T],
				$$
				where 
				$$
				\begin{aligned}
					\ds \Xi(s,\a(s)) &=\boldsymbol{X}^{-1}(s)\big\{\big[C(s,\a(s))+D(s,\a(s)) \Theta(s,\a(s))\big]^{2}-A(s,\a(s))-B(s,\a(s)) \Theta(s,\a(s))\big\},\\
					\ns\ds \D(s,\a(s)) &=-\boldsymbol{X}^{-1}(s)\big[C(s,\a(s))+D(s,\a(s)) \Theta(s,\a(s))\big].
				\end{aligned}
				$$
				Applying It\^o's formula to the right-hand side of \rf{21.9.6.1}, we have
				\begin{align*}
					d P(s,\a(s))&=-F(s,\a(s), \BX(s),\BY(s),\BZ(s),\BU(s)) \boldsymbol{X}^{-1}(s) d s+Z(s) \Delta(s) d s
					+\boldsymbol{Z}(s) \boldsymbol{X}^{-1}(s) d W(s)\\
					&\q+\sum_{k,l=1}^D\boldsymbol{\G}_{kl}(s)\boldsymbol{X}^{-1}(s)d\widetilde{N}_{kl}(s)
					+\boldsymbol{Y}(s) \Xi(s,\a(s)) d s+\boldsymbol{Y}(s) \Delta(s,\a(s)) d W(s)\\
					&=\big[-A(s,\a(s))^{\top} P(s,\a(s))-C(s,\a(s))^{\top} \Pi(s,\a(s))-Q(s,\a(s))
					-S(s,\a(s))^{\top} \Theta(s,\a(s))\\
					\ns\ds &\q+P(s,\a(s))\big[(C(s,\a(s))+D(s,\a(s)) \Theta(s,\a(s)))^{2}
					-A(s,\a(s))-B(s,\a(s)) \Theta(s,\a(s))\big]\\
					&\q -\Pi(s,\a(s))(C(s,\a(s))+D(s,\a(s)) \Theta(s, \a(s))\big] d s+\sum_{k,l=1}^D\z_{kl}(s)d\widetilde{N}_{kl}(s)\\
					&\q +\big[\Pi(s,\a(s))-P(s,\a(s))[C(s,\a(s))+D(s,\a(s)) \Theta(s,\a(s))]\big] d W(s)\\
					\ns\ds &=-\big[\hat{Q}(s,\a(s))
					+\hat{S}(s, \a(s))^\top \Th(s,\a(s))\big] ds
					%
					+\Lambda(s)d W(s)+\z(s)\bullet d\widetilde{N}(s).
				\end{align*}							
				Note that $P(s,\a(s))$ is symmetric, i.e., $P(s,\a(s))=P(s,\a(s))^{\top}$ for $s\in[0,T]$.
				Comparing the diffusion coefficients of the above BSDEs satisfied by $P(\cd,\a(\cd))$ and $P^{\top}(\cd,\a(\cd))$, we obtain
				$
				\Lambda(s)=\Lambda(s)^{\top}$ for  $s\in[0,T]$.
				Further combining \rf{21.9.3.5} and \rf{21.9.6.2}, we get
				\begin{align*}
					&	\hat{S}(s,\a(s))+\hat{R}(s,\a(s))\Theta(s,\a(s))
					=F(s,\a(s), \BX(s), \BY(s), \BZ(s), \BU(s))\boldsymbol{X}^{-1}(s)%
					=0.
				\end{align*}
				%
				%
				This completes the proof.
			\end{proof}

			\begin{lemma}\label{21.9.8.3} \sl
				Suppose {\rm (\textbf{H})} holds, and there exists a constant $\e>0$ such that  \rf{21.8.21.3.2} holds. Then
				\begin{align}\label{21.9.6.3}
					\hat{R}(s,\a(s))=R(s,\a(s))+D(s,\a(s))^\top P(s,\a(s))D(s,\a(s))\ges\e I_m,\q \hb{a.e. on $[0,T]$, a.s..}
				\end{align}
			\end{lemma}

			\begin{proof}
				The procedure of proof can be divided into three steps.
				
				\ms
				
				\textbf{Step 1}.  Let us temporarily assume that processes $\Theta(\cd)=\{\Theta(s,\a(s))\}_{s\in[0,T]}$ and $\Lambda(\cd)=\{\Lambda(s)\}_{s\in[0,T]}$ defined by \rf{21.9.6.2} satisfy the following condition:
				\begin{equation}\label{21.9.7.2}
					\esssup_{\omega \in \Omega} \int_{0}^{T}\big[|\Theta(s, \omega,\a(s))|^{2}
					+|\Lambda(s, \omega)|^{2}\big] d s<\infty.
				\end{equation}
				Choosing an arbitrary control $v (\cd)\in \mathcal{U}[0, T]$, we consider the following SDE
				\begin{equation}\label{21.9.6.4}
					\left\{\begin{aligned}
						\ds d X^v(s) =& \big\{\big[A(s,\a(s))+B(s,\a(s)) \Theta(s,\a(s))\big] X^v(s)+B(s,\a(s)) v(s)\big\} d s \\
						\ns\ds &+ \big\{\big[C(s,\a(s))+D(s,\a(s)) \Theta(s,\a(s))\big] X^v(s)+D(s,\a(s)) v(s)\big\} d W(s), \q  s \in[0, T], \\
						\ns\ds X^v(0) =& 0,\q~\a(0)=i_0.
					\end{aligned}\right.
				\end{equation}
				By the standard SDE theory, we see that $X^v(\cd)$, solution of SDE \rf{21.9.6.4}, belongs to space $L_{\dbF}^{2}(\Omega ; C([0, T] ; \mathbb{R}^{n}))$. Hence, 							%
				\begin{equation}\label{21.9.6.8}
					u(s) \triangleq \Theta(s,\a(s)) X^v(s)+v(s) \in \mathcal{U}[0, T].
				\end{equation}
				According to $u(\cd)$ defined in \rf{21.9.6.8} and the uniqueness of SDEs, $X(\cd)$, solution of equation \rf{state} w.r.t. the initial triple $(0,0,i_0)$ over $[0,T]$,  %
				coincides with $X^v(\cd)$, solution of \rf{21.9.6.4}, i.e.,
				\begin{equation}\label{22.6.25.3}
					X^v(s)=X(s),\q s\in[0,T].
				\end{equation}
				Note that $P(s,\a(s))$ satisfies BSDE \rf{21.9.3.4}. By applying It\^o's formula to $\big\langle P(s,\a(s)) X^v(s), X^v(s)\big\rangle$ firstly and then taking expectations on both sides, we can get
				\begin{equation}\label{22.6.25.5}
					\begin{aligned}
						\mathbb{E}\big\langle G(\a(T)) X^v(T), X^v(T)\big\rangle&=\dbE\int_0^T\Big\{-\big\langle \big[Q(s,\a(s)) +\hat{S}(s,\a(s))^{\top}\Theta(s,\a(s))\big] X^v(s), X^v(s)\big\rangle
						\\
						&\qq\qq\ \ +2\big\langle\big[\hat{S}(s,\a(s))-S(s,\a(s))\big]^\top u(s), X^v(s,\a(s))\big\rangle\\
						\ns\ds &\qq\qq\ \ +\big\langle D(s,\a(s))^{\top} P(s,\a(s)) D(s,\a(s)) u(s), u(s)\big\rangle\Big\} d s.
					\end{aligned}
				\end{equation}
				%
				Substituting \rf{22.6.25.5} into the cost functional \rf{cost} with the initial triple $(0,0,i_0)$ and noting that \rf{22.6.25.3}, we have
				$$
				\begin{aligned}
					J(0,0,i_0 ; u(\cd))
					&=\dbE\int_0^T\Big\{
					-\big\langle\hat{S}(s,\a(s))^{\top}\Theta(s,\a(s))X^v(s), X^v(s)\big\rangle
					\\
					&\qq\q\q\ \ +2\big\langle \hat{S}(s,\a(s))^\top u(s), X^v(s)\big\rangle+\big\langle\hat{R}(s,\a(s))u(s), u(s)\big\rangle\Big\} d s.
				\end{aligned}
				$$
				Combining \rf{21.9.7.1} and \rf{21.9.6.8}, we further deduce that
				\begin{align*}
					\ds J(0,0,i_0 ; u(\cd))
					&=\mathbb{E} \int_{0}^{T}\big\langle\hat{R}(s,\a(s))\big[u(s)-\Theta(s,\a(s)) X^v(s)\big], u(s)-\Theta(s,\a(s)) X^v(s)\big\rangle d s\\
					&=\mathbb{E} \int_{0}^{T}\big\langle\hat{R}(s,\a(s))v(s), v(s)\big\rangle d s.
				\end{align*}
				Finally, \rf{21.8.21.3.2} implies that
				$$J(0,0,i_0; u(\cd)) \geqslant 0,\q \forall u(\cd) \in \mathcal{U}[0, T].$$
				Therefore, we conclude from the above equation that%
				\begin{equation}\label{21.9.7.3}
					\hat{R}(s,\a(s))=R(s,\a(s))+ D(s,\a(s))^{\top} P(s,\a(s)) D(s,\a(s))\geqslant 0,
					\ \text { a.e. on }[0, T], \text { a.s.}.
				\end{equation}

				\textbf{Step 2}. Now we prove that without the additional condition \rf{21.9.7.2}, the above result \rf{21.9.7.3} still holds. The key method is to apply a localization technique so that the preceding argument can be applied to a certain stopped SLQ problem. In detail,  for each $k \geqslant 1$, we define the following stopping time (with the convention $\inf \varnothing=\infty)$:
				$$
				\tau_{k}=\inf \Big\{t \in[0, T] ; \int_{0}^{t}\big[|\Theta(s,\a(s))|^{2}+|\Lambda(s)|^{2}\big] d s \geqslant k\Big\} \wedge T.
				$$
				Take an arbitrary control $v(\cd) \in \mathcal{U}[0, T]$ and consider the state equation \rf{21.9.6.4} over the interval $\left[0, \tau_{k}\right]$.
				By the definition of $\tau_{k}$, we have
				$$
				\int_{0}^{\tau_{k}}\left[|\Theta(s,\a(s))|^{2}+|\Lambda(s)|^{2}\right] d s \leqslant k,
				$$
				which implies that $X^v(\cd)$, solution of \rf{21.9.6.4} over $[0,\t_k]$, belongs to  $L_{\dbF}^{2}\left(\Omega ; C\left(\left[0, \tau_{k}\right] ; \mathbb{R}^{n}\right)\right)$. Hence,
				$$
				u(s) \triangleq \Theta(s,\a(s)) X^v(s)+v(s) \in \mathcal{U}[0, \tau_{k}].
				$$
				Then we can proceed as in Step 1 to get that
				\begin{align*}
					\ds J(0,0,i_0 ; u(\cd))
					=\mathbb{E} \int_{0}^{\t_k}\big\langle\big[R(s,\a(s))+ D(s,\a(s))^{\top} P(s,\a(s)) D(s,\a(s))\big] v(s), v(s)\big\rangle d s.
				\end{align*}
				By the proof of the assertion (i) of \autoref{21.8.31.3}, we have that for any choosing $i_0\in\cS$,
				$$ J(0, 0,i_0 ; u(\cd)) \geqslant 0, \q \forall u(\cd) \in \mathcal{U}[0, \tau_{k}],$$
				and note that $v(\cd) \in \mathcal{U}[0, T]$ is arbitrary, so we have
				\begin{equation}\label{21.9.7.5}
					\hat{R}(s,\a(s))=R(s,\a(s))+ D(s,\a(s))^{\top} P(s,\a(s)) D(s,\a(s))
					\geqslant 0, \q \text {a.e. on }\left[0, \tau_{k}\right], \text { a.s..}
				\end{equation}
				Due to that the process $\boldsymbol{X}^{-1}(\cd)$ is continuous, $\boldsymbol{U}(\cd)$ and $\boldsymbol{Z}(\cd)$ are square-integrable, and the processes $P(\cd,\a(\cd))$, $C(\cd)$ and $D(\cd)$ are bounded,  from \rf{21.9.6.2} we have
				$$
				\int_{0}^{T}\left[|\Theta(s,\a(s))|^{2}+|\Lambda(s)|^{2}\right] d s<\infty, \ \text { a.s.}.
				$$
				This implies that $\lim _{k \rightarrow \infty} \tau_{k}=T$ almost surely.
				Then result \rf{21.9.7.3} still holds by letting $k \rightarrow \infty$ in \rf{21.9.7.5}.
				
				\ms
				
				\textbf{Step 3}.
				In order to obtain the stronger property \rf{21.9.6.3}, we take an arbitrary but fixed $\epsilon \in(0, \e)$ and consider the following stochastic LQ  problem of minimizing
				$$
				\begin{aligned}
					\ds J_{\epsilon}(t,\xi,\vf ; u(\cd))
					=&\mathbb{E}_t\bigg[\langle G(\a(T) X(T), X(T)\rangle\\
					\ns\ds &+\int_{t}^{T}\left\langle\left(\begin{array}{cc}
						Q(s,\a(s)) & S(s,\a(s))^{\top} \\
						S(s,\a(s)) & R(s,\a(s))-\epsilon I_{m}
					\end{array}\right)\left(\begin{array}{c}
						X(s) \\
						u(s)
					\end{array}\right),\left(\begin{array}{c}
						X(s) \\
						u(s)
					\end{array}\right)\right\rangle d s\bigg],
				\end{aligned}
				$$
				where $X(\cd)$ is the solution of state equation \rf{state}. Clearly, with $\e$ replaced by $\e-\epsilon$, the conditions of \autoref{21.9.3.1} still hold for the new cost functional $J_{\epsilon}(t,\xi,\vf ; u(\cd))$. Thus, there is a process
				$P_{\epsilon}(\cd)$ such that
				$$
				V_{\epsilon}(t,\xi,\vf) \triangleq \inf _{u(\cd) \in \mathcal{U}[t, T]} J_{_{\epsilon}}(t,\xi,\vf ; u(\cd))=\left\langle P_{\epsilon}(t,\vf) \xi, \xi\right\rangle, \q  \forall(t,\xi,\vf) \in[0, T] \times L_{\sF_{t}}^{\infty}(\Omega ; \mathbb{R}^{n})\ts
				L_{\sF_{t}^\a}^{2}(\Omega ; \cS).
				$$
				Then, by the previous discussion, we have
				$$
				R(s,\a(s))-\epsilon I_{m}+D(s,\a(s))^{\top} P_{\epsilon}(s,\a(s)) D(s,\a(s)) \geqslant 0,  \q \text{a.e. on }[0, T], \text { a.s.}.
				$$
				Now, by the definition of $J_{\epsilon}(t,\xi,\vf ; u(\cd))$, we deduce that
				%
				\begin{align*}
					\ds V(t,\xi,\vf)&=\inf _{u(\cd) \in \mathcal{U}[t, T]} J(t,\xi,\vf ; u(\cd)) \\
					\ns\ds &\geqslant \inf _{u(\cd) \in \mathcal{U}[t, T]} J_{\epsilon}(t,\xi,\vf ; u(\cd))=V_{\epsilon}(t,\xi,\vf), \q  \forall(t,\xi,\vf) \in[0, T] \times L_{\sF_{t}}^{\infty}\left(\Omega ; \mathbb{R}^{n}\right)\ts L_{\sF_{t}^\a}^{2}(\Omega ; \cS),
				\end{align*}
				from which we see that
				$$P(t,\vf) \geqslant P_{\epsilon}(t,\vf),\q \forall t\in[0,T],\ \vf\in L_{\sF_{t}^\a}^{2}(\Omega ; \cS),$$
				and therefore
				\begin{align*}
					&\hat{R}(s,\a(s))\geqslant R(s,\a(s))+D(s,\a(s))^{\top} P_{\epsilon}(s,\a(s)) D(s,\a(s)) \geqslant \varepsilon I_{m}, \ \text { a.e. on }[0, T], \text { a.s.. }
				\end{align*}
				%
				%
				Finally, note that $\epsilon \in(0, \e)$ is arbitrary, so property \rf{21.9.6.3} holds. This completes the proof.
			\end{proof}
			
			Based on the above preparations, we now can prove \autoref{21.9.3.1} and \autoref{21.11.26.1}.
			
			\begin{proof}[Proof of \autoref{21.9.3.1}]
				On one hand, from \autoref{21.9.8.2}, we see that the bounded process $P(\cd,\a(\cd))$ in \rf{21.8.21.6} and the processes $\L(\cd)$ and $\zeta(\cd)$ defined by \rf{21.9.6.2} satisfy BSDE \rf{21.9.3.4} and relation \rf{21.9.7.1}.
				On the other hand, \autoref{21.9.8.3} implies that
				$$\hat{R}(s,\a(s))=R(s,\a(s))+D(s,\a(s))^\top P(s,\a(s))D(s,\a(s))\ges\e I_m,\q \hb{a.e. on $[0,T]$, a.s.,}$$
				which, together with relation \rf{21.9.7.1}, deduces that
				\begin{equation}\label{21.9.8.4}
					\begin{aligned}
						\Theta(s,\a(s))&=
						-\hat{R}(s,\a(s))^{-1}\hat{S}(s,\a(s)),\q
						\text {a.e. on }[0, T], \text { a.s..}
					\end{aligned}
				\end{equation}
				Substituting \rf{21.9.8.4} into BSDE \rf{21.9.3.4} yields that
				\begin{equation}\label{21.9.11.1}
					\left\{\begin{aligned}
						\ds d P(s,\a(s))=&\ [\hat{\Psi}(s,\a(s))-\hat{Q}(s,\a(s))] d s+\Lambda(s) d W(s)
						+\z(s)\bullet d\wt{N}(s), \q s \in[0, T], \\
						\ns\ds P(T,\a(T))=&\ G(\a(T)), \q \a(0)=i_0,
					\end{aligned}\right.
				\end{equation}
				where $\hat{Q}(\cd)$ is defined by \eqref{hatsq}, and
				$$\hat{\Psi}(s,\a(s))\deq \hat{S}(s,\a(s))^\top\hat{R}(s,\a(s))^{-1}\hat{S}(s,\a(s)),\q s\in[0,T].$$
				Then SRE \rf{21.8.23.6} follows easily from \rf{21.9.11.1}, and $(\hat P(\cd,\a(\cd)),\hat\L(\cd),\hat\zeta(\cd))$, solution of SRE \rf{21.8.23.6}, coincide with $( P(\cd,\a(\cd)),\L(\cd),\zeta(\cd))$, solution of BSDE \rf{21.9.11.1}, i.e.,
				\begin{equation}\label{22.6.26.6}
					\hat P(s,\a(s))=P(s,\a(s)),\q \hat\L(s)=\L(s),\q \hat\zeta(s)=\zeta(s),\q s\in[0,T].
				\end{equation}
				In the following, we still adopt $( P(\cd,\a(\cd)),\L(\cd),\zeta(\cd))$ in order to keep the consistency of symbols.
				
				\ms
				
				It remains to prove that the processes $\Lambda(\cd)$  and $\zeta(\cd)$ are square-integrable.
				Note that in BSDE \rf{21.9.11.1}, the matrix-valued processes $A(\cd)$, $C(\cd)$,  $Q(\cd)$ and $P(\cd,\a(\cd))$ are all bounded and the process $\hat{\Psi}(\cd)$ is positive semi-definite, so we can choose a positive constant $K$ such that
				\begin{equation}\label{21.9.9.2}
					\left\{\begin{aligned}
						\ds	&\operatorname{tr}[P(s,\a(s))]+|P(s,\a(s))|^{2}  \leqslant K, \qq
						\operatorname{tr}[\hat{Q}(s,\a(s))]  \leqslant K[1+|\Lambda(s)|], \\
						\ns\ds	&\operatorname{tr}[P(s,\a(s)) \hat{Q}(s,\a(s))]  \leqslant|P(s,\a(s)) |\cdot | \hat{Q}(s,\a(s))|
						\leqslant K[1+|\Lambda(s)|] ,\\
						\ns\ds 	&\operatorname{tr}[-P(s,\a(s)) \hat{\Psi}(s,\a(s))]  \leqslant \lambda_{\max }[-P(s,\a(s))] \operatorname{tr}[\hat{\Psi}(s,\a(s))] \leqslant K \operatorname{tr}[\hat{\Psi}(s,\a(s))],
					\end{aligned}\right.
				\end{equation}
				for Lebesgue-almost every $s, \mathbb{P}$-a.s.. In the last inequality, we adopt Theorem 7.4.1.1 of Horn--Johnson \cite{Horn-Johnson2012}. 
				In the following, we denote by the same letter $K$ a generic positive constant whose value may be different from line to line. Define for each $m \geqslant 1$ the stopping time (with the convention $\inf \varnothing=\infty)$
				\begin{equation}\label{22.6.13.1}
					\lambda_{m}=\inf \Big\{t \in[0, T] ; \int_{0}^{t}\big[|\Lambda(s)|^{2}
					+\sum_{k,l=1}^D|\zeta_{kl}(s)|^2\l_{kl}(s)I_{\{\a(s-)=k\}} \big] d s \geqslant m\Big\},
				\end{equation}
				which implies that $\lim _{m \rightarrow \infty} \lambda_{m}=\infty$ almost surely.
				Then we have							
				\begin{equation}\label{21.9.9.1}
					\begin{aligned}
						P\left(t \wedge \lambda_{m},\a(t \wedge \lambda_{m})\right)=&\
						P(0,i_0)+\int_{0}^{t \wedge \lambda_{m}}[\hat{\Psi}(s,\a(s))-\hat{Q}(s,\a(s))] d s\\
						& +\int_{0}^{t \wedge \lambda_{m}} \Lambda(s) d W(s)
						+\int_{0}^{t \wedge \lambda_{m}}\zeta(s)\bullet d\wt{N}(s).
					\end{aligned}
				\end{equation}
				From the definition of $\l_m$, it is easy to see that the following processes
				$$
				\begin{aligned}
					\Big\{\int_{0}^{t \wedge \lambda_{m}} \Lambda(s) d W(s)\Big\}_{t\in[0,T]}&=\ \Big\{\int_{0}^{t} \Lambda(s) \mathbf{1}_{\left\{s \leqslant \lambda_{m}\right\}} d W(s)\Big\}_{t\in[0,T]},\\
					\Big\{\int_{0}^{t \wedge \lambda_{m}}\zeta(s)\bullet d\wt{N}(s)\Big\}_{t\in[0,T]}&=\ \Big\{\int_{0}^{t} \zeta(s)
					\mathbf{1}_{\left\{s \leqslant \lambda_{m}\right\}}\bullet d\wt{N}(s)\Big\}_{t\in[0,T]},
				\end{aligned}$$
				are matrix of square-integrable martingales w.r.t. the filtration $\dbF=\{\sF_t; 0\les t<\i\}$
				and $\dbF^\a=\{\sF^\a_t; 0\les t<\i\}$, respectively.
				Therefore, taking expectations on both side of \rf{21.9.9.1}, we have
				\begin{align*}
					\ds \dbE\big[ P\left(t \wedge \lambda_{m},\a(t \wedge \lambda_{m})\right)\big]=
					P(0,i_0)+\dbE\int_{0}^{t \wedge \lambda_{m}}\big[\hat{\Psi}(s,\a(s))-\hat{Q}(s,\a(s))\big] d s.
				\end{align*}
				Thus, combining with \rf{21.9.9.2}, we see that
				\begin{align}\label{21.9.9.3}
					\mathbb{E} \int_{0}^{t \wedge \lambda_{m}} \operatorname{tr}[\hat{\Psi}(s,\a(s))] d s
					&=\mathbb{E} \operatorname{tr}\big[P\left(t \wedge \lambda_{m},\a(t \wedge \lambda_{m})\right)-P(0,i_0)\big]
					+\mathbb{E} \int_{0}^{t \wedge \lambda_{m}} \operatorname{tr}[\hat{Q}(s,\a(s))] d s \nn\\
					& \leqslant K\Big[1+\mathbb{E} \int_{0}^{t \wedge \lambda_{m}}|\Lambda(s)| d s\Big].
				\end{align}
				On the other hand, for BSDE \rf{21.9.11.1}, applying It\^o's formula to $P(s,\a(s))^2$ and denoting $\zeta^2(s)\deq(\zeta_{kl}(s)^2)$, we have
				\begin{align*}
					d[P(s,\a(s))]^{2}=& \Big\{P(s,\a(s))\big[\hat{\Psi}(s,\a(s))-\hat{Q}(s,\a(s))\big]
					+\big[\hat{\Psi}(s,\a(s))-\hat{Q}(s,\a(s))\big] P(s,\a(s))\\
					& +\Lambda(s)^{2}+\sum_{k,l=1}^D\zeta^2_{kl}(s)\l_{kl}(s)I_{\{\a(s-)=k\}} \Big\}ds
					+\big[P(s,\a(s)) \Lambda(s)+\Lambda(s) P(s,\a(s))\big] d W(t)\\
					&+ \sum_{k,l=1}^D 2P(s-)\zeta_{kl}(s)d\widetilde{N}_{kl}(s)
					+\zeta^2(s) \bullet d\widetilde{N}(s).
				\end{align*}
				Now, note \rf{22.6.13.1}, a similar argument shows that
				\begin{align*}
					\dbE[P(s,\a(s))]^{2}=&P(0,i_0)+ \dbE\int_0^{t \wedge \lambda_{m}}\Big\{P(s,\a(s))\big[\hat{\Psi}(s,\a(s))-\hat{Q}(s,\a(s))\big] \\
					&  +\big[\hat{\Psi}(s,\a(s))-\hat{Q}(s,\a(s))\big]P(s,\a(s)) +\Lambda(s)^{2}
					+\sum_{k,l=1}^D\zeta_{kl}(s)^2\l_{kl}(s)I_{\{\a(s-)=k\}} \Big\} ds.
				\end{align*}
				Combining with \rf{21.9.9.2}-\rf{21.9.9.3} and recalling the Frobenius norm, we have
				\begin{align*}
					&\mathbb{E} \int_{0}^{t \wedge \lambda_{m}}|\Lambda(s)|^{2} d s
					+\mathbb{E} \int_{0}^{t \wedge \lambda_{m}}
					\sum_{k,l=1}^D|\zeta_{kl}(s)|^2\l_{kl}(s)I_{\{\a(s-)=k\}}  d s\\
					=\,&\operatorname{tr}\bigg[\mathbb{E} \int_{0}^{t \wedge \lambda_{m}}[\Lambda(s)]^{2} d s
					+\mathbb{E} \int_{0}^{t \wedge \lambda_{m}}
					\sum_{k,l=1}^D[\zeta_{kl}(s)]^2\l_{kl}(s)I_{\{\a(s-)=k\}}  d s\bigg]\\
					=\,&\mathbb{E}\left|P\left(t\wedge\lambda_{m},\a(t\wedge\lambda_{m})\right)\right|^{2}-|P(0,i_0)|^{2}
					+2 \mathbb{E} \int_{0}^{t \wedge \lambda_{m}} \operatorname{tr}[P(s,\a(s)) \hat{Q}(s,\a(s))] d s\\
					\,& +2 \mathbb{E} \int_{0}^{t \wedge \lambda_{m}} \operatorname{tr}[-P(s,\a(s)) \hat{\Psi}(s,\a(s))] d s \\
					\leqslant \,&  K+K \mathbb{E} \int_{0}^{t \wedge \lambda_{m}}[1+|\Lambda(s)|] d s
					+K \mathbb{E} \int_{0}^{t \wedge \lambda_{m}} \operatorname{tr}[\hat{\Psi}(s,\a(s))] d s \\
					\leqslant \,& K+K\mathbb{E} \int_{0}^{t \wedge \lambda_{m}}|\Lambda(s)| d s
					\les  K+2 K^{2}+\frac{1}{2} \mathbb{E} \int_{0}^{t \wedge \lambda_{m}}|\Lambda(s)|^{2} d s,
				\end{align*}
				where in the last inequality we employ the Cauchy-Schwarz's inequality. Hence, we obtain that
				\begin{align*}
					\frac{1}{2}\mathbb{E} \int_{0}^{t \wedge \lambda_{m}}|\Lambda(s)|^{2} d s
					+\mathbb{E} \int_{0}^{t \wedge \lambda_{m}}
					\sum_{k,l=1}^D|\zeta_{kl}(s)|^2\l_{kl}(s)I_{\{\a(s-)=k\}} d s
					\les K+2 K^{2}.
				\end{align*}
				Due to that $\lim_{m\rightarrow\i}\l_m=\i$ almost surely and $K$ does not depend on $m$ and $t$, we conclude that the processes $\L(\cd)$ and $\zeta(\cd)$ are square-integrable by letting $m\rightarrow\i$ firstly and then $t\uparrow T$.
			\end{proof}

			\begin{proof}[Proof of \autoref{21.11.26.1}]
				Note that \autoref{21.8.21.4} and \autoref{21.8.21.2} imply that  Problem (M-SLQ) is uniquely open-loop solvable at any initial time $t<T$, and the open-loop optimal control $u^{*}(\cd)$ w.r.t. $(t, \xi,\vf)$ is given by
				$$u^{*}(s)=\left(u_{1}(s), \cdots, u_{n}(s)\right) \xi , \q s\in[t,T].$$
				Therefore, for the sake of the optimal control w.r.t. any initial triple $(t, \xi,\vf) \in \mathcal{D}$, it is sufficient to determine the open-loop optimal control $u_{j}(\cd)=\left\{u_{j}(s) ; s\in[0,T]\right\}$ w.r.t. $\left(0, e_{j},i_0\right)$ for each $j=1, \ldots, n$.
				From \autoref{21.9.2.1}, we see that the process $\boldsymbol{X}(\cd)=\{\boldsymbol{X}(s) ; 0 \leqslant s \leqslant T\}$ is invertible. On the other hand,
				\autoref{21.9.8.2} tells us that finding the open-loop optimal controls $u_{1}, \cdots, u_{n}$ are equivalent to finding
				$$\Theta(s,\a(s))=\boldsymbol{U}(s) \boldsymbol{X}(s)^{-1}, \q s\in[0,T].$$
				The latter can be accomplished by solving SRE \rf{21.8.23.6}, whose solvability is obtained by \autoref{21.9.3.1}. In fact, from the proof of \autoref{21.9.3.1}, we see that $\Theta(\cd)$ is actually determined by \rf{21.9.8.4}.
				Summarizing these observations, we conclude the closed-loop representation \rf{22.6.7.1}.
			\end{proof}

			\section{Proof of the auxiliary results}\label{ProofPri}
			
			In this section, we prove the results of  \autoref{Sec3.2}. Inspired by Mou--Yong \cite{Mou-Yong-06} and Chen--Yong \cite{Chen-Yong-01}, we present a representation of the cost functional, which is characterized as a bilinear form on a Hilbert space in terms of the adapted solutions of some forward-backward stochastic differential equations, using the technique of It\^{o}'s formula with jumps.
			Let us present a simple lemma first.

			\begin{lemma}\label{Lemma3.1}\sl
				Under {\rm (\textbf{H})}, for any initial triple $(t,\xi,\vf)\in\cD$ and a control $u(\cd)\in\cU[t,T]$, we have
				\begin{equation}\label{3.1}
					\begin{aligned}
						{J}(t,\xi,\vf;u(\cd))=&\lan Y(t),\xi \ran
						+\dbE\Big[\int_t^T\big\lan 	F(s,\a(s), X(s), Y(s), Z(s), u(s)), u(s)\big\ran ds\Big|\sF_t\Big],
					\end{aligned}
				\end{equation}
				where the quadruple $(X(\cd),Y(\cd),Z(\cd),\Gamma(\cd))$ is the solution of SDE \rf{state} and BSDE \rf{21.8.20.1}.
				%
			\end{lemma}
			
			\begin{proof}
				The proof is trivial by applying It\^{o}'s formula to $\lan Y(\cd),X(\cd)\ran$ on $[t,T]$. In fact, using It\^{o}'s formula to
				$s\mapsto \lan Y(s),X(s)\ran$ and then taking the conditional expectation yields that
				\bel{3.3}\ba{ll}
				\ds \dbE\big[\lan G(\a(T))X(T),X(T)\ran|\sF_t\big]=\lan Y(t),\xi \ran
				+\dbE\Big[\int_t^T\big[\big\lan B(s,\a(s))^\top Y(s)+D(s,\a(s))^\top Z(s)\\
				\ns\ds \qq\qq\qq\hspace{2.8cm}-S(s,\a(s))X(s),u(s)\big\ran-\big\lan Q(s,\a(s))X(s),X(s)\big\ran\big] ds\Big|\sF_t\Big].
				\ea\ee
				Then substituting \rf{3.3} into the cost functional ${J}(t,\xi,\vf;u(\cd))$ leads to \rf{3.1}.
			\end{proof}

			Next, we let $(\tilde{X}(\cd),\tilde{Y}(\cd),\tilde{Z}(\cd),\tilde{\Gamma}(\cd))$ and
			$(\bar{X}(\cd),\bar{Y}(\cd),\bar{Z}(\cd),\bar{\Gamma}(\cd))$ be the adapted solutions of the following decoupled FBSDEs, respectively,
			\begin{align}
				\label{3.4}&\left\{\ba{ll}
				\ds d\tilde{X}(s)=\big[A(s,\a(s))\tilde{X}(s)+B(s,\a(s))u(s)\big]ds
				+\big[C(s,\a(s))\tilde{X}(s)+D(s,\a(s))u(s)\big]dW(s),\\
				\ns\ds d\tilde{Y}(s)=-\widetilde{F}(s,\a(s), \tilde{X}(s), \tilde{Y}(s), \tilde{Z}(s), u(s))ds+\tilde{Z}(s)dW(s)+
				\tilde{\G}(s)\bullet d\wt{N}(s),\q s\in[t,T],\\
				\ns\ds \tilde{X}(t)=0,\q\a(t)=\vf,\q  \tilde{Y}(T)=G(\a(T))\tilde{X}(T),
				\ea\right.\\
				\label{3.5}&\left\{\ba{ll}
				\ds d\bar{X}(s)=A(s,\a(s))\bar{X}(s)ds+C(s,\a(s))\bar{X}(s)dW(s),\q s\in[t,T],\\
				\ns\ds d\bar{Y}(s)=-\widetilde{F}_0(s,\a(s), \bar{X}(s),  \bar{Y}(s), \bar{Z}(s))ds
				+\bar{Z}(s)dW(s)+\bar{\G}(s)\bullet d\widetilde{N}(s),\\
				\ns\ds \bar{X}(t)=\xi,\q\a(t)=\vf,\q  \bar{Y}(T)=G(\a(T))\bar{X}(T).
				\ea\right.
			\end{align}
			%
			Then the adapted solution $(X(\cd),Y(\cd),Z(\cd),\Gamma(\cd))$ of SDE \rf{state} and BSDE \rf{21.8.20.1} could be written as the sum of $(\tilde{X}(\cd),\tilde{Y}(\cd),\tilde{Z}(\cd),\tilde{\Gamma}(\cd))$ and $(\bar{X}(\cd),\bar{Y}(\cd),\bar{Z}(\cd),\bar{\Gamma}(\cd))$, i.e.,
			$$X(s)=\tilde X(s)+\bar X(s),\q~ Y(s)=\tilde Y(s)+\bar Y(s),\q~
			Z(s)=\tilde Z(s)+\bar Z(s),\q~ \G(s)=\tilde \G(s)+\bar \G(s),\q s\in[t,T].$$
			%
			In what follows, for any $u(\cd),v(\cd)\in\cU[t,T]$, we set 
			%
			\begin{align*}
				[[ u, v ]]_t=\dbE_t\int_t^T\langle u(s),v(s)\rangle ds,\quad  [[ u, v ]]=\dbE\int_t^T\langle u(s),v(s)\rangle ds.
			\end{align*}%
			%
			Now, for any $u(\cd)\in\cU[t,T]$ and $\xi\in\cX_t\deq L^2_{\sF_t}(\O;\dbR^n)$,  we define two linear operators  $\cN_t$  and  $\cL_t$ as follows: for $s\in[t,T]$,
			\begin{align}
				\label{3.6}
				[\cN_tu](s,\a(s))=F(s,\a(s), \tilde{X}(s), \tilde{Y}(s), \tilde{Z}(s), u(s)), \q	[\cL_t\xi](s,\a(s))=F_0(s,\a(s), \bar{X}(s),  \bar{Y}(s), \bar{Z}(s)).
			\end{align}
			%
			
			Then, for these two operators, we have the following result.
			\begin{lemma}\label{Proposition3.2} \sl
				Under the condition {\rm(\textbf{H})}, the following two assetions hold:
				\begin{itemize}
					\item [{\rm (i)}] On the Hilbert space, the linear operator $\cN_t$ defined in \rf{3.6} is a bounded self-adjoint operator.
					\item [{\rm (ii)}]
					From the Hilbert space $\cX_t$ into the Hilbert space $\cU[t,T]$, the linear operator $\cL_t$ defined in \rf{3.6} is a bounded operator. Moreover, there exists a positive constant $K$, independent of $(t,\xi,\vf)$, such that
					\bel{3.8}[[\cL_t\xi,\cL_t\xi]]\les K\dbE|\xi|^2,\q \forall \xi\in\cX_t.\ee
				\end{itemize}
			\end{lemma}

			\begin{proof}
				(i) Using the same argument as in Proposition 2.1 of Sun--Yong \cite{Sun-Yong-14}, it is easy to obtain the boundedness of the operator $\cN_t$.
				In order to show that $\cN_t$ is self-adjoint, it suffices to prove that for any $u_1(\cd),u_2(\cd)\in\cU[t,T]$, the following relation holds:
				\begin{align}\label{3.9}
					\dbE\int_t^T\big\langle [\cN_t u_1](s,\a(s)),u_2(s)\big\rangle ds
					=\dbE\int_t^T\big\langle u_1(s),[\cN_t u_2](s,\a(s))\big\rangle ds.
				\end{align}
				For this, we let
				$(\tilde{X}_j(\cd),\tilde{Y}_j(\cd),\tilde{Z}_j(\cd),\tilde{\Gamma}_j(\cd))$ with $j=1$, $2$ be the adapted solution of FBSDE \rf{3.4} with $u(\cd)$ is replaced by $u_j(\cd)$. On  one hand, applying It\^{o}'s formula to $\lan \tilde{Y}_2(s),\tilde{X}_1(s)\ran$ first and then taking the expectation, we have
				\begin{equation*}
					\begin{aligned}
						\dbE\big\langle G(\a(T))\tilde{X}_2(T),\tilde{X}_1(T)\big\rangle
						=&\ \dbE\int_t^T\big[\big\langle B(s,\a(s))^\top \tilde{Y}_2(s)+D(s,\a(s))^\top\tilde{Z}_2(s),u_1(s)\big\rangle\\
						& - \big\langle Q(s,\a(s))\tilde{X}_2(s),\tilde{X}_1(s)\big\rangle
						- \big\langle S(s,\a(s))\tilde{X}_1(s),u_2(s)\big\rangle\big]ds.
					\end{aligned}
				\end{equation*}
				On the other hand, applying It\^{o}'s formula to
				$\lan \tilde{Y}_1(s),\tilde{X}_2(s)\ran$ first and then taking the expectation lead to
				\begin{equation*}
					\begin{aligned}
						\dbE\big\langle G(\a(T))\tilde{X}_1(T),\tilde{X}_2(T)\big\rangle
						=&\ \dbE\int_t^T\big[\big\langle B(s,\a(s))^\top \tilde{Y}_1(s)+D(s,\a(s))^\top\tilde{Z}_1(s),u_2(s)\big\rangle\\
						& - \big\langle Q(s,\a(s))\tilde{X}_1(s),\tilde{X}_2(s)\big\rangle
						- \big\langle S(s,\a(s))\tilde{X}_2(s),u_1(s)\big\rangle\big]ds.
					\end{aligned}
				\end{equation*}
				Combining the above two equations and noting that $G(\cd)$ and $Q(\cd)$ are symmetric, we have
				\begin{equation*}
					\begin{aligned}
						\dbE\int_t^T\big\langle [\cN_t u_1](s,\a(s))-R(s,\a(s))u_1(s),u_2(s)\big\rangle ds=\dbE\int_t^T\big\langle u_1(s), [\cN_t u_2](s,\a(s))-R(s,\a(s))u_2(s)\big\rangle ds.
					\end{aligned}
				\end{equation*}%
				%
				Thus, from the fact that $R(\cd)$ is also symmetric, we can easily obtain  \rf{3.9}.
				
				\ms
				
				(ii) It suffices to prove \rf{3.8}. Choose a positive constant $\rho_1$ such that
				\bel{3.12}|G(\a(T))|^2,|B(s,\a(s))|^2,|D(s,\a(s))|^2,|S(s,\a(s))|^2,|Q(s,\a(s))|^2\les \rho_1,
				\q \hb{a.e.}, s\in[0,T],\hb{ a.s.}.\ee
				Then using the vector inequality $|\b_1+\cdot\cdot\cdot+\b_n|\les n(|\b_1|^2+\cdot\cdot\cdot+\b_n|^2)$,
				we have
				\begin{equation*}
					\begin{aligned}
						\ [[\cL_t\xi,\cL_t\xi]]=\dbE\int_t^T\big|F_0(s,\a(s), \bar{X}(s), \bar{Y}(s), \bar{Z}(s))\big|^2ds		%
						\les \ 3\rho_1\dbE\int_t^T\[|\bar{Y}(s)|^2+|\bar{Z}(s)|^2|+\bar{X}(s)|^2\]ds.
					\end{aligned}
				\end{equation*}
				By the classical theory of SDEs and BSDEs (see also Proposition 2.1 of Sun--Yong \cite{Sun-Yong-14}), we have that for the decoupled FBSDE \rf{3.5}, there exists a positive constant $\rho_2$, independent of $(t,\xi,\vf)$, such that
				\begin{align}
					\dbE\int_t^T\big[|\bar{Y}(s)|^2+|\bar{Z}(s)|^2\big]ds&\les
					\rho_2\dbE\Big[|G(\a(T))\bar X(T)|^2+\int_t^T|Q(s,\a(s))\bar X(s)|^2ds\Big], \nn\\
					\dbE|\bar X(T)|^2+\dbE\int_t^T|\bar X(s)|^2ds&\les \rho_2\dbE|\xi|^2. \label{3.14}
				\end{align}
				Finally, combining \rf{3.12} and \rf{3.14}, we deduce that $[[\cL_t\xi,\cL_t\xi]]\les3\rho_1(\rho_1\rho_2^2+\rho_2)\dbE|\xi|^2$ for all $\xi\in\cX_t$.
			\end{proof}

			\begin{remark}\label{Remark3.3} \rm
				Denote by the quadruple $(\cX(\cd),\cY(\cd),\cZ(\cd),\breve{\Gamma}(\cd))$ the adapted solution of the FBSDE:
				$$\left\{\ba{ll}
				\ds d\cX(s)=A(s,\a(s))\cX(s)ds+C(s,\a(s))\cX(s)dW(s),\\
				\ns\ds d\cY(s)=-\widetilde{F}_0(s,\a(s), \cX(s), \cY(s), \cZ(s))ds+\cZ(s)dW(s)+\breve{\G}(s)\bullet d\wt{N}(s),\q s\in[0,T],\\
				\ns\ds \cX(0)=I_n,\q\a(0)=i_0,\q  \cY(T)=G(\a(T))\cX(T).
				\ea\right.$$
				It is easy to check that the process $\cX(\cd)$ is invertible (denote by $\cX^{-1}(\cd)$) with $\cX^{-1}(\cd)$ satisfying the following equation:
				$$\left\{\ba{ll}
				\ds d\cX^{-1}(s)=\cX^{-1}(s)\big[C(s,\a(s))^2-A(s,\a(s))\big]ds
				-\cX^{-1}(s)C(s,\a(s))dW(s),\q s\in[0,T],\\
				\ns\ds \cX^{-1}(0)=I_n,\q\a(0)=i_0.
				\ea\right.$$
				For every $\xi\in L^\i_{\sF_t}(\Omega;\dbR^n)$ and $\eta(s)\deq (\eta_{kl}(s))\in \cM_D(\mathbb{R}^{n\times n})$,  we set  $\eta(s)\circ\xi\deq(\eta_{kl}(s)\xi)$. 
				Now, on the other hand,  the following processes
				\begin{align*}
					\cX(s)\cX^{-1}(t)\xi, \q~ \cY(s)\cX^{-1}(t)\xi, \q~
					\cZ(s)\cX^{-1}(t)\xi, \q~ \breve{\G}(s)\circ\cX^{-1}(t)\xi, \q~ s\in[t,T],
				\end{align*}
				are all square-integrable and satisfy FBSDE \rf{3.5}.  Therefore, by the uniqueness of the adapted solutions, one has that for $s\in[t,T]$,
				\begin{align*}
					\bar X(s)= \cX(s)\cX^{-1}(t)\xi, \q \bar Y(s)=\cY(s)\cX^{-1}(t)\xi, \q
					\bar Z(s)= \cZ(s)\cX^{-1}(t)\xi, \q \bar\G(s)=\breve{\G}(s)\circ\cX^{-1}(t)\xi.
				\end{align*}
				Hence, for every $\xi\in L^\i_{\sF_t}(\Omega;\dbR^n)$, the operator $\cL_t\xi$ could be represented in terms of the quadruple $(\cX(\cd),\cY(\cd),\cZ(\cd),\breve{\Gamma}(\cd))$ as follows,
				\bel{3.15}[\cL_t\xi](s,\a(s))
				=F_0(s,\a(s), \cX(s), \cY(s), \cZ(s))\cX^{-1}(t)\xi,\q s\in[t,T].\ee%
				This relation is useful in proving \autoref{21.8.22.4}.
			\end{remark}
			
			Now, we are ready to show the representation to the cost functional $J_0(t,\xi,\vf; u(\cd))$.
			Note that $J_0(t,\xi,\vf ; u(\cd))$ and ${J}(t,\xi,\vf ; u(\cd))$ have the relation $J_0(t, \xi,\vf ; u(\cd))=\mathbb{E} {J}(t,\xi,\vf ; u(\cd))$, and from \autoref{21.8.19.1}, the first component $M(\cd)$ of  BSDE \rf{2.3} is bounded.
			
			\begin{proposition}\label{Theorem3.4}\sl
				Let $\cN_t$ and $\cL_t$ be defined in \rf{3.6} and $M(\cd)$ be the first component of the solution of BSDE \rf{2.3}. Then, under (\textbf{H}), the cost functional $J(t,\xi,\vf;u(\cd))$ has the following representation:
				\bel{3.16-et}J(t,\xi,\vf;u(\cd))=[[\cN_tu, u]]_t+2[[\cL_t\xi,u]]_t+\langle M(t)\xi,\xi\rangle,\
				\hb{ for all }(t,\xi,\vf)\in\cD.\ee
				%
				%
				Furthermore, the cost functional $J_0(t,\xi,\vf;u(\cd))$ has the following representation:
				\bel{3.16}J_0(t,\xi,\vf;u(\cd))=[[\cN_tu, u]]+2[[\cL_t\xi,u]]+\dbE\langle M(t)\xi,\xi\rangle,\
				\hb{ for all }(t,\xi,\vf)\in\cD.\ee
				%
				%
			\end{proposition}
			\begin{proof}
				Fix any $(t,\xi,\vf)\in\cD$ and $u(\cd)\in\cU[t,T]$. Let $(\tilde{X}(\cd),\tilde{Y}(\cd),\tilde{Z}(\cd),\tilde \G(\cd))$, $(\bar{X}(\cd),\bar{Y}(\cd),\bar{Z}(\cd),\bar\G(\cd))$,
				and $(X(\cd),Y(\cd),Z(\cd),\G(\cd))$
				be the adapted solutions of FBSDE \rf{3.4}, FBSDE \rf{3.5}, and SDE \rf{state} and BSDE \rf{21.8.20.1}, respectively.
				Then
				$$X(s)=\tilde X(s)+\bar X(s),\q~ Y(s)=\tilde Y(s)+\bar Y(s),\q~
				Z(s)=\tilde Z(s)+\bar Z(s),\q~ \G(s)=\tilde \G(s)+\bar \G(s),\q s\in[t,T].$$
				By \autoref{Lemma3.1}, the definitions of $\cN_t$ and $\cL_t$, we have
				\bel{3.17}\ba{ll}
				\ds J(t,\xi,\vf;u(\cd))=\lan \tilde{Y}(t),\xi \ran+\lan \bar{Y}(t),\xi \ran
				+\dbE_t\int_t^T\big\lan [\cN_tu](s)+[\cL_t\xi](s),u(s)\big\ran ds.
				\ea\ee
				On one hand, applying It\^{o}'s formula to $\lan \tilde Y(s),\bar X(s)\ran$ on $[t,T]$ implies that
				\begin{align*}
					\dbE_t\lan G(\a(T))\tilde X(T),\bar X(T)\ran-\lan \tilde Y(t),\xi\ran=
					-\dbE_t\int_t^T\[\lan Q(s,\a(s))\tilde X(s),\bar X(s)\ran+\lan S(s,\a(s))\bar X(s),u(s)\ran\]ds.
				\end{align*}
				On the other hand, applying It\^{o}'s formula to $\lan \bar Y(s),\tilde X(s)\ran$ on $[t,T]$ gives that
				\begin{align*}
					\dbE_t\lan G(\a(T))\bar X(T),\tilde X(T)\ran=
					\dbE_t\int_t^T\[\lan B(s,\a(s))^\top\bar Y(s)+D(s,\a(s))^\top\bar Z(s),u(s)\ran
					-\lan Q(s,\a(s))\bar X(s),\tilde X(s)\ran\]ds.
				\end{align*}
				So we have
				\begin{align*}
					\lan\tilde{Y}(t),\xi\ran
					=\dbE_t\int_t^T\[\lan F_0(s,\a(s), \bar{X}(s),  \bar{Y}(s), \bar{Z}(s)), u(s)\ran\]ds
					=\dbE_t\int_t^T\lan[\cL_t\xi](s),u(s)\ran ds. 
				\end{align*}
				Moreover, applying It\^{o}'s formula to $M(s)\bar X(s)$, we have
				%
				\begin{align}
					d[M(s)\bar{X}(s)]&=-\Big[A(s,\a(s))^\top M(s) \bar{X}(s)
					+C(s,\a(s))^\top\big(M(s)C(s,\a(s))+\Phi(s)\big) \bar{X}(s)+Q(s,\a(s))\bar{X}(s)\Big]ds  \nonumber\\
					\label{21.9.15.1}  &\q\ +\big(M(s)C(s,\a(s))+\Phi(s)\big)\bar{X}(s)dW(s)
					+(\eta(s)\circ\bar X(s))\bullet d\widetilde{N}(s).	
				\end{align}
				Since $M(T)\bar X(T)=\bar Y(T)$, we see that the triple $(M\bar X,(MC+\Phi)\bar X,\eta\circ \bar X)$ satisfies the same BSDE as $(\bar Y(\cd),\bar Z(\cd),\bar \G(\cd))$. Thus, by
				comparing equations \rf{3.5} and \rf{21.9.15.1}, we have
				\begin{align*}
					\bar Y(s)=M(s)\bar X(s),\q\bar Z(s)=\big(M(s)C(s,\a(s))+\Phi(s)\big) \bar{X}(s),\q
					\bar\G(s)=\eta(s)\circ\bar X(s),\q s\in[t,T].
				\end{align*}
				It follows that $\lan\bar Y(t),\xi\ran=\lan M(t)\xi,\xi\ran$.
				Substituting this relation  into \rf{3.17} implies result \rf{3.16-et}. By taking expectation to \rf{3.16-et}, we can easily obtain \rf{3.16}.
			\end{proof}

			\begin{corollary}\label{Corollary3.5}\sl
				Under the condition (\textbf{H}), the following two assertions hold:
				\begin{itemize}
					\item [{\rm (i)}] A control $u^*(\cd)\in\cU[t,T]$ is optimal for Problem (M-SLQ)$_0$ w.r.t. $(t,\xi,\vf)\in\cD$ if and only if
					\begin{align}\label{3.19}
						\cN_t\ges 0,\q \hb{and}\q\cN_tu^*+\cL_t\xi=0.
					\end{align}
					\item [{\rm (ii)}] If $\cN_t$ is invertible and satisfies the positivity condition $\cN_t\ges0$, then Problem (M-SLQ)$_0$ is uniquely solvable at $t$, and the unique open-loop optimal control $u^*(\cd)$ w.r.t. $(t,\xi,\vf)\in\cD$ is given by
					\begin{align*}
						u^*(s)=-[\cN^{-1}_t\cL_t\xi](s),\q   s\in[t,T].
					\end{align*}
				\end{itemize}
			\end{corollary}

			\begin{proof}
				By \autoref{def-open}, we see that $u^*(\cd)\in\cU[t,T]$ is optimal for Problem (M-SLQ)$_0$ w.r.t. the initial triple $(t,\xi,\vf)$ if and only if
				\bel{3.20}J_0(t,\xi,\vf;u^*(\cd)+\l v(\cd))-J_0(t,\xi,\vf;u^*(\cd))\ges0,\q \forall v(\cd)\in\cU[t,T],\  \l\in\dbR.\ee
				%
				According to the representation  \rf{3.16}, one has
				\begin{align*}
					J_0(t,\xi,\vf;u^*(\cd)+\l v(\cd))-J_0(t,\xi,\vf;u^*(\cd))
					%
					=&~\l^2[[\cN_tv,v]]+2\l[[\cN_tu^*,v]]+2\l[[\cL_t\xi,v]].
				\end{align*}
				Hence, the equation \rf{3.20} is equivalent to
				\begin{align*}
					\l^2[[\cN_tv,v]]+2\l[[\cN_tu^*+\cL_t\xi,v]]\ges0,\q \forall v(\cd)\in\cU[t,T],\ \forall\l\in\dbR.
				\end{align*}
				%
				%
				Therefore, we must have
				\begin{align*}
					[[\cN_tv,v]]\ges0 \q \hb{and}\q [[\cN_tu^*+\cL_t\xi,v]]=0, \q \forall v\in\cU[t,T],
				\end{align*}
				which implies that \rf{3.19} holds. 
				The converse assertion is obvious. The assertion (ii) is a direct consequence of the assertion (i). This completes the proof.
			\end{proof}

			\begin{remark}\label{remark5.6}\rm
				It is noteworthy that in the assertion (ii) of \autoref{Corollary3.5}, the assumption that
				$\cN_t\ges0$  and $ \cN_t$  is invertible
				are equivalent to that $\cN_t$ is uniformly positive, i.e., there exists a positive constant $\e$ such that
				\begin{align}\label{3.21}
					[[\cN_tu,u]]\ges\e[[u,u]],\q \forall u(\cd)\in\cU[t,T].
				\end{align}
				From the assertion (i) of \autoref{Corollary3.5}, we see that the condition $\cN_t\ges0$
				(or equivalently, $[[\cN_tu,u]]\ges0$ for every $u(\cd)\in\cU[t,T]$) is necessary
				for the existence of an open-loop optimal control of Problem (M-SLQ)$_0$, and from the assertion (ii) of \autoref{Corollary3.5},
				we see that the condition \rf{3.21}, slightly stronger than $\cN_t\ges0$,
				is sufficient for the existence of an open-loop optimal control of Problem (M-SLQ)$_0$.
				Moreover, by \rf{3.16}, we have
				$[[\cN_tu,u]]=J_0(t,0,\vf;u(\cd)).$
			\end{remark}

			Based on the above discussions, we are in a position to prove the auxiliary results of \autoref{Sec3.2}.

			\begin{proof}[Proof of \autoref{Theorem4.1}]
				By the assertion (i) of \autoref{Corollary3.5}, we see that $u^*(\cd)\in\cU[t,T]$ is an open-loop optimal control of Problem (M-SLQ)$_0$ with respect to $(t,\xi,\vf)$ if and only if \rf{3.19} holds. According to \rf{3.16}, we see that $\cN_t\ges0$ is equivalent to
				$$J_0(t,0,\vf;u(\cd))=[[\cN_tu,u]]\geq 0,\q \forall u(\cd)\in\cU[t,T],$$
				which is exactly condition \rf{21.9.15.2}. From the definitions of $\cN_t$ and $\cL_t$, it is easy to verify that \begin{align*}
					[\cN_tu^*+\cL_t\xi](s,\a(s))=	F(s,\a(s), X^*(s),  Y^*(s), Z^*(s), u^*(s))
					, \q s\in[t,T],
				\end{align*}
				where $(X^*(\cd), Y^*(\cd),Z^*(\cd),\G^*(\cd))$ is the adapted solution of FBSDE \rf{4.1}. Therefore, $\cN_tu^*+\cL_t\xi=0$ is equivalent to stationarity condition \rf{4.2}.
			\end{proof}

			\begin{proof}[Proof of \autoref{Theorem4.2}]
				The sufficiency is obvious. Next we prove the necessity.
				Suppose that $u^*(\cd)\in\cU[t,T]$  is optimal for Problem (M-SLQ)$_0$, and let $(X^*(\cd),Y^*(\cd),Z^*(\cd),\G^*(\cd))$ be the adapted solution of FBSDE \rf{4.1}. In order to prove that $u^*(\cd)$ is also optimal for Problem ${(\hb{M-SLQ})}$, it suffices to show that for any set $\Lambda\in\sF_t$,
				\begin{align}\label{4.3}
					\dbE[L(t,\xi,\vf;u^*(\cd))\mathbf{1}_\L]\les\dbE[L(t,\xi,\vf;u(\cd))\mathbf{1}_\L],\q \forall u(\cd)\in\cU[t,T],
				\end{align}
				where for any $(t,\xi,\vf)\in\cD$ and $u(\cd)\in\cU[t,T]$,
				\bel{cost-2}\ba{ll}
				\ds L(t,\xi,\vf;u(\cd))\deq\big\lan G(\a(T))X(T),X(T)\big\ran
				+\int_t^T\llan\begin{pmatrix}Q(s,\a(s))&S(s,\a(s))^\top\\S(s,\a(s))&R(s,\a(s))\end{pmatrix}
				\begin{pmatrix}X(s)\\ u(s)\end{pmatrix},
				\begin{pmatrix}X(s)\\ u(s)\end{pmatrix}\rran ds.
				\ea\ee
				Note that in \rf{cost-2}, $X(\cd)$ is the solution of the state equation \rf{state} with respect to $u(\cd)$.
				For this, we would like to fix an arbitrary control $u(\cd)\in\cU[t,T]$ and an arbitrary set $\L\in\sF_t$. Define
				\begin{align*}
					\widehat\xi(\o)=\xi(\o)\mathbf{1}_\L(\o),\q
					\widehat\vf(\o)=\vf(\o)\mathbf{1}_\vf(\o),\q
					\widehat u(s)=u(s)\mathbf{1}_\L(\o),\q
					\widehat u^*(s)=u^*(s)\mathbf{1}_\L(\o),
				\end{align*}
				and consider the following FBSDE:
				\bel{4.4}\left\{\ba{ll}
				d\widehat{X}^*(s)=[A(s,\a(s))\widehat{X}^*(s)+B(s,\a(s))\widehat u^*(s)]ds
				+[C(s,\a(s))\widehat{X}^*(s)+D(s,\a(s))\widehat u^*(s)]dW(s),\\
				\ns\ns\ds d\widehat{Y}^*(s)=-\widetilde{F}(s,\a(s), \widehat{X}^*(s), \widehat{Y}^*(s), \widehat{Z}^*(s), \widehat u^*(s))ds+\widehat{Z}^*(s)dW(s)+\widehat{\G}^*(s)\bullet d\widetilde{N}(s),\q s\in[t,T],\\
				\ns\ns\ds \widehat{X}^*(t)=\widehat \xi,\q\a(t)=\widehat\vf,\q  \widehat{Y}^*(T)=G(\a(T))\widehat{X}^*(T).
				\ea\right.\ee
				It is straightforward to verify that the solution $(\widehat{X}^*(\cd),\widehat{Y}^*(\cd),\widehat{Z}^*(\cd),\widehat{\G}^*(\cd))$
				of FBSDE \rf{4.4} is given by
				\begin{align*}
					\widehat{X}^*(s)=X^*(s)\mathbf{1}_\L(\o), \q
					\widehat{Y}^*(s)=Y^*(s)\mathbf{1}_\L(\o),\q
					\widehat{Z}^*(s)=Z^*(s)\mathbf{1}_\L(\o), \q
					\widehat{\G}^*(s)=\G^*(s)\mathbf{1}_\L(\o).
				\end{align*}
				%
				%
				Now, from \autoref{Theorem4.1}, the quadruple $(X^*(\cd),Y^*(\cd),Z^*(\cd),\G^*(\cd))$ satisfies condition \rf{4.2}. Hence,  multiplying on both sides of \rf{4.2} by $\mathbf{1}_\L$, one has
				\begin{align*}
					F(s,\a(s), \widehat{X}^*(s), \widehat{Y}^*(s), \widehat{Z}^*(s), \widehat u^*(s))
					=0,\q \hb{a.e. $s\in[t,T]$, a.s.}.
				\end{align*}
				Again, applying  \autoref{Theorem4.1} to the initial triple $(t,\widehat\xi,\widehat\vf)$, we conclude that $\widehat u^*(\cd)$ is an open-loop optimal control of Problem (M-SLQ)$_0$ w.r.t. the initial triple $(t,\widehat\xi,\widehat\vf)$. Therefore,
				$$\dbE[L(t,\widehat \xi,\widehat\vf;\widehat u^*(\cd))]
				\les\dbE[L(t,\widehat\xi,\widehat\vf;\widehat u(\cd))].$$
				Note that the state process $X^*(\cd)=X^*(\cd;t,\xi,\vf,u^*(\cd))$  and
				the state process $\widehat{X}^*(\cd)=X^*(\cd;t,\widehat\xi,\widehat\vf,\widehat u^*(\cd))$ are related by
				$X^*(\cd;t,\xi,\vf,u^*(\cd))\mathbf{1}_\L=X^*(\cd;t,\widehat\xi,\widehat\vf,\widehat u^*(\cd)).$
				It follows that $L(t,\xi,\vf;u^*(\cd))\mathbf{1}_\L=L(t,\widehat\xi,\widehat\vf;\widehat u^*(\cd))$. Similarly, we have $L(t,\xi,\vf;u(\cd))\mathbf{1}_\L=L(t,\widehat\xi,\widehat\vf;\widehat u(\cd))$. Thus,
				$$\dbE[L(t,\xi,\vf;u^*(\cd))\mathbf{1}_\L]=\dbE[L(t,\widehat\xi,\widehat\vf;\widehat u^*(\cd))]
				\les \dbE[L(t,\widehat\xi,\widehat\vf;\widehat u(\cd))]=\dbE[L(t,\xi,\vf;u(\cd))\mathbf{1}_\L],$$
				and from which we see that \rf{4.3} holds. This completes the proof.
			\end{proof}

			\begin{proof}[Proof of \autoref{21.9.3.3}] \rm
				Note that $(X^*(\cd),u^*(\cd))$ is an open-loop optimal pair with respect to $(t,\xi,\vf)$, so \autoref{Theorem4.1} implies
				\begin{align*}
					F(s,\a(s), X^*(s),  Y^*(s), Z^*(s), u^*(s))=
					0,\q \hb{a.e. $s\in[t,T]$, a.s..}
				\end{align*}
				Then the result follows immediately from \autoref{Lemma3.1}.
			\end{proof}

			\begin{proof}[Proof of \autoref{Corollary4.5}]\rm
				Let $\t\in\cT[t,T]$. According to \autoref{Theorem4.1} and \autoref{Theorem4.2}, it suffices
				to show that
				$J_0(\t,0,\a(\t);u(\cd))\ges0$ for all $u(\cd)\in\cU[t,T]$, and
				the adapted solution $(X^*_\t(\cd),Y^*_\t(\cd),Z^*_\t(\cd),\G^*_\t(\cd))$ to FBSDE \rf{4.1} with $u^*(\cd)$ replaced by $u^*(\cd)|_{[\t,T]}$
				satisfies the following stationarity condition:
				$$B(s,\a(s))^\top Y^*_\t(s)+D(s,\a(s))^\top Z^*_\t(s)
				+S(s,\a(s))X^*_\t(s)+R(s,\a(s))u^*(s)|_{[\t,T]}=0,\ \ \hb{a.e. $s\in[t,T]$, a.s..}$$
				For this, we let $u(\cd)\in\cU[\t,T]$ be an arbitrary and define the zero-extension of $u(\cd)$ over the time interval $[t, T ]$ as follows:
				\begin{equation}\label{22.6.16.1}
					u_e(s)=\left\{\ba{ll}
					\ds 0,\qq\ s\in[t,\t),\\
					\ns\ds u(s),\q s\in[\t,T].
					\ea\right.
				\end{equation}
				Clearly, $u_e(\cd)$ belongs to the space $\cU[t,T]$. Denote by $X^\t$ and $X^t$ the solutions to the following SDEs, respectively,

				\begin{equation*}
					\left\{\begin{aligned}
						\ds dX^\t(s)=& [A(s,\a(s))X^\t(s)+B(s,\a(s))u(s)]ds\\
						\ns\ds & +[C(s,\a(s))X^\t(s)+D(s,\a(s))u(s)]dW(s),\q s\in[\t,T],\\
						\ns\ds X^\t(\t)=& 0, \q \a(t)=\vf,
					\end{aligned}\right.
				\end{equation*}
				and
				\begin{equation*}
					\left\{\begin{aligned}
						\ds dX^t(s)=& [A(s,\a(s))X^t(s)+B(s,\a(s))u_e(s)]ds\\
						\ns\ds & +[C(s,\a(s))X^t(s)+D(s,\a(s))u_e(s)]dW(s),\q s\in[t,T],\\
						\ns\ds X^t(t)=& 0, \q\a(t)=\vf.
					\end{aligned}\right.
				\end{equation*}
				Since the initial states of the above two SDEs are 0, and note that $v_e = 0$ on $[t,\t)$, we have
				$$X^t(s)=\left\{\ba{ll}
				\ds 0,\qq\q s\in[t,\t],\\
				\ns\ds X^\t(s),\q s\in[\t,T],
				\ea\right.$$
				from which one has that
				\begin{align}
					\ds &J_0(\t,0,\a(\t);u(\cd)) \nn\\
					\ns\ds &=\dbE\bigg[\big\lan G(\a(T))X^\t(T),X^\t(T)\big\ran
					+\int_\t^T\llan\begin{pmatrix}Q(s,\a(s))&S(s,\a(s))^\top\\S(s,\a(s))&R(s,\a(s))\end{pmatrix}
					\begin{pmatrix}X^\t(s)\\ u(s)\end{pmatrix},
					\begin{pmatrix}X^\t(s)\\u(s)\end{pmatrix}\rran ds\bigg] \nonumber\\
					\ns\ds &=\dbE\bigg[\big\lan G(\a(T))X^t(T),X^t(T)\big\ran
					+\int_t^T\llan\begin{pmatrix}Q(s,\a(s))&S(s,\a(s))^\top\\S(s,\a(s))&R(s,\a(s))\end{pmatrix}
					\begin{pmatrix}X^t(s)\\ u_e(s)\end{pmatrix},
					\begin{pmatrix}X^t(s)\\u_e(s)\end{pmatrix}\rran ds\bigg] \nonumber\\
					\ns\ds &=J_0(t,0,\vf;u_e(\cd)).  \label{4.5}
				\end{align}
				Due to that under the conditions, Problem (M-SLQ) is solvable at $(t,\xi,\vf)$, and from the assertion (i) of \autoref{Theorem4.1} and relation \rf{4.5}, we obtain that
				$$J_0(\t,0,\a(\t);u(\cd))=J_0(t,0,\vf;u_e(\cd))\ges0,\q\forall u(\cd)\in\cU[\t,T].$$
				
				Next, note that the quadruple $(X^*(\cd),Y^*(\cd),Z^*(\cd),\G^*(\cd))$ is the adapted solution of FBSDE \rf{4.1}.
				Since $u^*(\cd)\in\cU[t,T]$ is an open-loop optimal control w.r.t. $(t,\xi,\vf)$,  by assertion (ii) of \autoref{Theorem4.1},
				$$B(s,\a(s))^\top Y^*(s)+D(s,\a(s))^\top Z^*(s)
				+S(s,\a(s))X^*(s)+R(s,\a(s))u^*(s)=0,\q \hb{a.e. $s\in[t,T]$, a.s..}$$
				Then the results hold from the fact that
				$$(X^*_\t(s),Y^*_\t(s),Z^*_\t(s),\G^*_\t(s))=(X^*(s),Y^*(s),Z^*(s),\G^*(s)),\q s\in[\t,T].$$
				This completes the proof.
			\end{proof}

			\begin{proof}[Proof of \autoref{21.8.21.4}]
				For arbitrary $u(\cd) \in \mathcal{U}[\tau, T]$, we define $ u_{e}(\cd)$ as in \rf{22.6.16.1}.
				%
				%
				Using the same argument as in the proof of \autoref{Corollary4.5} with $t=0$, one has
				$$
				J_0(\tau, 0,\a(\t) ; u(\cd))=J_0\left(0,0,i_0; u_{e}(\cd)\right) \geqslant \e \mathbb{E} \int_{0}^{T}\left|u_{e}(s)\right|^{2} d s=\e \mathbb{E} \int_{\tau}^{T}|u(s)|^{2} d s.
				$$
				Hence, by setting $\t\equiv t$ and from \autoref{remark5.6}, the assertion (ii) of \autoref{Corollary3.5} and \autoref{Theorem4.2}, we obtain that Problem (M-SLQ) is uniquely solvable.
			\end{proof}

			\begin{proof}[Proof of \autoref{21.8.21.2}]
				By \autoref{Theorem4.1}, it is easy to see that the first assertion holds.
				Now, we consider the second assertion. Note that $\xi\in L_{\sF_{t}}^{\infty}\left(\Omega ; \mathbb{R}^{n}\right)$ is bounded, so the pair
				$(\mathbb X(s), \mathbb U(s)) \triangleq(\boldsymbol{X}(s) \xi, \boldsymbol{U}(s) \xi)$ 
				is square-integrable and satisfies the following state equation:
				\begin{equation}\label{22.7.7.1}
					\left\{\begin{aligned}
						\ds d \mathbb X(s) &=\big[A(s,\a(s)) \mathbb X(s)+B(s,\a(s)) \mathbb U(s)\big] d s\\
						\ns\ds &\q +\big[C(s,\a(s)) \mathbb X(s)+D(s,\a(s)) \mathbb U(s)\big] d W(s), \q s \in[t, T],\\
						\ns\ds \mathbb X(t) &=\xi,\q \a(t)=\vf.
					\end{aligned}\right.
				\end{equation}
				Similarly, we see that the triple
				$\left(\mathbb Y(s), \mathbb  Z(s),\breve{\mathbf{\G}}(s)\right) \triangleq(\boldsymbol{Y}(s) \xi, \BZ(s) \xi,\boldsymbol{\G}(s)\circ\xi )$
				is the adapted solution to the following adjoint BSDE associated with $\left(\mathbb X(\cd), \mathbb U(\cd)\right)$:
				$$\left\{\ba{ll}
				\ds d\mathbb Y(s)=-\widetilde{F}(s,\a(s), \mathbb X(s), \mathbb Y(s),  \mathbb Z(s), \mathbb U(s))ds+\mathbb Z(s)dW(s)+\breve{\mathbf{\G}}(s)\bullet d\widetilde{N}(s), \q s\in[t,T],\\
				%
				%
				\ns\ns\ds \mathbb Y(T)=G(\a(T))\mathbb X(T),\q \a(t)=\vf.
				\ea\right.$$
				Furthermore, \rf{21.8.20.3} deduces that
				\begin{align*}
					& F(s,\a(s), \mathbb X(s),  \mathbb Y(s),\mathbb  Z(s),\mathbb U(s))
					=F(s,\a(s), \BX(s), \BY(s), \BZ(s), \BU(s)) \xi
					=0,\q \hb{a.e. $s\in[t,T]$, a.s.}.
				\end{align*}
				Thus, combining \autoref{Theorem4.1}, the pair $\left(\mathbb X(\cd), \mathbb U(\cd)\right)$ is optimal w.r.t. the initial triple $(t,\xi,\vf)$.
			\end{proof}

			\begin{proof}[Proof of \autoref{21.8.21.5}]
				Let  $\left\{\left(X_{j}(s), u_{j}(s)\right) \right\}_{s\in[t,T]}$ and $\{(\boldsymbol{X}(s), \boldsymbol{U}(s)) \}_{s\in[t,T]}$ be as in  \autoref{21.8.21.2},
				then  the pair $\{(\boldsymbol{X}(s)\xi, \boldsymbol{U}(s)\xi) \}_{s\in[t,T]}$  is optimal w.r.t. the initial triple $(t,\xi,\vf)$. 
				For simplicity presentation,  we denote
				$$
				\mathbf{M}(T,\a(T))\deq \boldsymbol{X}(T)^{\top} G(\a(T)) \boldsymbol{X}(T), \quad
				\boldsymbol{N}(s,\a(s))\deq \left(\begin{array}{c}
					\boldsymbol{X}(s) \\
					\boldsymbol{U}(s)
				\end{array}\right)^{\top}
				\left(\begin{array}{cc}
					Q(s,\a(s)) & S(s,\a(s))^{\top} \\
					S(s,\a(s)) & R(s,\a(s))
				\end{array}\right)
				\left(\begin{array}{c}
					\boldsymbol{X}(s) \\
					\boldsymbol{U}(s)
				\end{array}\right),
				$$
				and then we could rewrite
				\bel{22.7.11.0}
				L(t,\xi,\vf ; \boldsymbol{U} \xi)
				%
							=\langle \mathbf{M}(T,\a(T))\xi,\xi\rangle+\int_{t}^{T}\langle\boldsymbol{N}(s,\a(s))\xi,\xi\rangle ds,
							\ee
							where $L(\cd)$ is defined in \rf{cost-2}.
							Now, on one hand, since the pair $\{(\boldsymbol{X}(s)\xi, \boldsymbol{U}(s)\xi) \}_{s\in[t,T]}$  is optimal w.r.t. the initial triple $(t,\xi,\vf)$, from the definition \rf{optim-2} we have
							\begin{equation}\label{22.7.11.1}
								{V}(t,\xi,\vf)= \mathbb{E}[L(t,\xi,\vf ; \boldsymbol{U} \xi) \mid \sF_{t}].
							\end{equation}
							On the other hand, note that the state equation \rf{22.7.7.1} works over the time horizon $[t,T]$ with the initial value  $\xi$ is $\sF_t$-measurable and $\a(t)=\vf$ is $\sF_t^\a$-measurable, hence there is a process  $P:[0,T]\ts\cS\ts\Omega\rightarrow\dbS^n$, which is $\dbF$-adapted, such that
							\begin{equation}\label{22.7.11.2}
								P(t,\a(t))= \mathbb{E}\Big[\mathbf{M}(T,\a(T))+\int_{t}^{T} \boldsymbol{N}(s,\a(s)) d s \Big| \sF_{t}\Big].
							\end{equation}
							In fact, for the special case of $C=G=1$ and $B=D=Q=S=R=0$,  one can calculate the solution of SDE \rf{22.7.7.1} to obtain that
							\begin{equation*}
								\mathbb{X}(s)=\xi\exp\left\{ \int_t^s\Big[A(r,\a(r))-\frac{1}{2}\Big]dr+W(s)-W(t)\right\},\q s\in[t,T].
							\end{equation*}
							Note that $\mathbb X(s)=\boldsymbol{X}(s) \xi$, and in this case $\mathbf{M}(T,\a(T))=\boldsymbol{X}(T)^\top\boldsymbol{X}(T)$ and $ \boldsymbol{N}(s,\a(s))\equiv0$, so
							\begin{equation*}
								\dbE[\mathbf{M}(T,\a(T))|\sF_t]=\dbE\bigg[\exp\bigg\{ \int_t^T[2A(r,\a(r))-1]dr\bigg\}\Big|\sF_t\bigg].
							\end{equation*}
							Thus \rf{22.7.11.2} holds due to that $\a(\cd)$ is a Markov chain.
							Moreover, for the general situation, one can still prove that \rf{22.7.11.2} holds using a similar argument.
							Finally, by combining \rf{22.7.11.0}-\rf{22.7.11.2} and noting that $\a(t)=\vf$, we have
							\begin{align*}
								{V}(t,\xi,\vf)=&\ \mathbb{E}[L(t,\xi,\vf ; \boldsymbol{U} \xi) \mid \sF_{t}]\\
								=&\ \Big\langle\mathbb{E}\Big[\mathbf{M}(T,\a(T))+\int_{t}^{T} \boldsymbol{N}(s,\a(s)) d s \Big| \sF_{t}\Big] \xi, \xi\Big\rangle = \langle P(t,\vf) \xi, \xi\rangle.
							\end{align*}
							This completes the proof.
						\end{proof}

						\begin{proof}[Proof of \autoref{21.8.22.4}]
							First, we prove the boundedness of the process $P=\{P(t,i); (t,i)\in[0,T]\ts \cS\}$.
							From  \autoref{21.8.21.4}, we have that for every $t \in[0, T)$, the operator $\mathcal{N}_{t}$ defined in \rf{3.6} satisfies
							\begin{align}\label{eq:ntuu}
								[[\mathcal{N}_{t} u, u ]]=J_0(t, 0,\vf; u(\cd))\ges
								\e \mathbb{E} \int_{t}^{T}|u(s)|^{2} d s=\e[[u,u]], \q \forall u(\cd) \in \mathcal{U}[t, T],
							\end{align}
							which implies that the operator $\mathcal{N}_{t}$ is positive and invertible.
							Moreover, on one hand, for any initial state $\xi \in L_{\sF_{t}}^{\infty}\left(\Omega ; \mathbb{R}^{n}\right)$,
							by assertion (ii) of \autoref{Corollary3.5}, the related open-loop optimal control is given by
							\begin{align*}
								u^*_{t,\xi}(s)=-[\cN^{-1}_t\cL_t\xi](s),\q   s\in[t,T].
							\end{align*}
							Then, by substituting $u_{t, \xi}^{*}(\cd)$ into  \rf{3.16}, one has 
							\begin{equation}\label{21.8.23.1}
								\dbE\langle P(t,\vf) \xi, \xi\rangle=V_0(t,\xi,\vf)=\dbE\langle M(t) \xi, \xi\rangle-[[ \mathcal{N}_{t}^{-1} \mathcal{L}_{t} \xi, \mathcal{L}_{t} \xi]],
							\end{equation}
							where $\cN_t$ and $\cL_t$ are defined in \rf{3.6} and $M(\cd)$ is the first component of the solution of BSDE \rf{2.3}. Therefore,
							\begin{equation}\label{21.8.23.2}
								\mathbb{E}\langle P(t,\vf) \xi, \xi\rangle \leqslant \mathbb{E}\langle M(t) \xi, \xi\rangle.
							\end{equation}
							On the other hand, combining \rf{eq:ntuu} and \rf{21.8.23.1} leads to 
							\begin{equation}\label{21.8.23.3}
								\mathbb{E}\langle P(t,\vf) \xi, \xi\rangle \geqslant
								\mathbb{E}\langle M(t) \xi, \xi\rangle-\e^{-1} [[ \mathcal{L}_{t} \xi, \mathcal{L}_{t} \xi ]].
							\end{equation}
							Thus, from assertion (ii) of \autoref{Proposition3.2}, we have
							\begin{equation}\label{21.8.23.4}
								\mathbb{E}\langle P(t,\vf) \xi, \xi\rangle
								\ges \mathbb{E}\langle M(t) \xi, \xi\rangle-\e^{-1} K \mathbb{E}|\xi|^{2}
								=\mathbb{E}\left\langle\left[M(t)-\e^{-1} K I_{n}\right] \xi, \xi\right\rangle,
							\end{equation}
							where $K$ is a positive constant comes from \autoref{Proposition3.2}.
							Note that $\xi\in L^\i_{\sF_t}(\Omega;\dbR^n)$ is bounded and arbitrary, which, together with \rf{21.8.23.2} and \rf{21.8.23.4}, implies that
							$$
							M(t)-\e^{-1} K I_{n} \leqslant P(t,\vf) \leqslant M(t),\q (t,\vf)\in[0,T]\ts L_{\sF_t^\alpha}^2(\Omega;\mathcal{S}).
							$$
							Therefore, the process $P=\{P(t,i); (t,i)\in[0,T]\ts \cS\}$ is bounded follows from the above inequality and the fact that $M(\cd)$ is bounded (see \autoref{21.8.19.1}).
							
							\ms
							
							In the following, we  prove that it is left-continuous. For simplicity but without loss of generality, we only study the left-continuity at $t=T$, and the situation of $t\in(0,T)$ could be proved similarly by considering the related Problem (M-SLQ). Note that, owing to \rf{21.8.23.2} and \rf{21.8.23.3}, for every initial triple
							$(t,\xi,\vf) \in[0,T)\times L_{\sF_{t}}^{\infty}\left(\Omega;\mathbb{R}^{n}\right)\ts L_{\sF_t^\alpha}^2(\Omega;\mathcal{S})$,
							we have
							\begin{equation}\label{21.8.23.5}
								\mathbb{E}\langle M(t) \xi, \xi\rangle-\e^{-1}[[\mathcal{L}_{t} \xi, \mathcal{L}_{t} \xi ]]\leqslant \mathbb{E}\langle P(t,\vf) \xi, \xi\rangle \leqslant \mathbb{E}\langle M(t) \xi, \xi\rangle.
							\end{equation}
							Noting representation \rf{3.15}, we can rewrite $[[\cL_t\xi,\cL_t\xi]]$ as follows,
							\begin{align}\label{22.7.7.2}
								[[ \mathcal{L}_{t} \xi, \mathcal{L}_{t} \xi ]]=
								\mathbb{E} \int_{t}^{T} \Big\langle  \big[F_0(s,\a(s), \cX(s), \cY(s), \cZ(s))
								\cX^{-1}(t)\big]^{\top} \big[F_0(s,\a(s), \cX(s), \cY(s), \cZ(s))
								\cX^{-1}(t)\big] \xi, \xi \Big\rangle d s .
							\end{align}
							Again, note that $\xi\in L^\i_{\sF_t}(\Omega;\dbR^n)$ is bounded and arbitrary, which, together with \rf{21.8.23.5} and \rf{22.7.7.2}, implies that for any $(t,\vf)\in[0,T]\ts L_{\sF_t^\alpha}^2(\Omega;\mathcal{S})$,
							\begin{align*}
								\ds &M(t) -\e^{-1}\big[\cX^{-1}(t)\big]^{\top}
								\int_{t}^{T}F_0(s,\a(s), \cX(s), \cY(s), \cZ(s))^\top F_0(s,\a(s), \cX(s), \cY(s), \cZ(s))ds\ \cX^{-1}(t)\\
								\ns\ds &\leqslant P(t,\vf) \leqslant M(t).
							\end{align*}
							Note that $\cX^{-1}(t)$ and $M(\cd)$ are $\sF_t$-measurable, and $P(\cd,\a(\cd))\in L_{\mathbb{F}}^{\infty}(0, T ; \mathbb{S}^{n})$.
							Hence, by taking the conditional expectations w.r.t. $\sF_t$ on both side of the above inequality, one has
							\begin{align*}
								\ds& M(t) -\e^{-1}\big[\cX^{-1}(t)\big]^{\top} \mathbb{E} \bigg[\int_{t}^{T}
								F_0(s,\a(s), \cX(s), \cY(s), \cZ(s))^\top F_0(s,\a(s), \cX(s), \cY(s), \cZ(s))
								d s \Big| \sF_{t}\bigg] \cX^{-1}(t)\\
								\ns\ds & \leqslant P(t,i) \leqslant M(t).
							\end{align*}
							Finally, letting $t \uparrow T$ and using the dominated convergence theorem, we get that
							$$
							\lim _{t \uparrow T} P(t,i)= M(T)=G(i)=P(T,i).
							$$
							Hence, the second purpose is obtained.
						\end{proof}

						\begin{proof}[Proof of \autoref{21.8.31.3}]
							First, we prove the assertion (i).
							For any fix initial triple $(\sigma,\xi,\vf)\in\cT[0,\tau)
							\ts L_{\sF_{\sigma}}^{\infty}(\Omega ; \mathbb{R}^{n})\ts
							L_{\sF_{\sigma}^\a}^{2}(\Omega ; \cS)$
							and $u(\cd) \in \mathcal{U}[\sigma, \tau]$, denote by
							$\left\{X_{1}(s)\right\}_{s\in [\sigma, \tau]}$
							the corresponding solution of state equation \rf{state} w.r.t. $(\sigma,\xi,\vf)$ over $[\sigma,\t]$.
							Consider Problem (M-SLQ) for the initial triple $\left(\tau, X_{1}(\tau),\a(\t)\right)$.
							By the assumptions, there exists $\e>0$ such that \rf{21.8.21.3.2} holds. Hence, from  \autoref{21.8.21.4} and  representation \rf{21.8.21.6},
							Problem (M-SLQ) is solvable at $\tau$ and
							$$
							{V}\left(\tau, X_{1}(\tau),\a(\t)\right)
							=\langle P(\tau,\a(\t)) X_{1}(\tau), X_{1}(\tau)\rangle .
							$$
							Now for Problem (M-SLQ), let $v^{*} (\cd)\in \mathcal{U}[\tau, T]$ be an open-loop optimal control w.r.t. the initial triple $\left(\tau, X_{1}(\tau),\a(\tau)\right)$ and let $\{X_{2}^{*}(s)\}_{s\in[\tau, T]}$ be the corresponding open-loop optimal state process. Define
							$$
							\left[u \oplus v^{*}\right](s)= \begin{cases}u(s), & s \in[\sigma, \tau), \\ v^{*}(s), & s \in[\tau, T],\end{cases}\qq \hb{and}\qq
							\widetilde{X}(s)= \begin{cases}X_{1}(s), & s \in[\sigma, \tau), \\ X_{2}^{*}(s), & s \in[\tau, T].\end{cases}
							$$
							Then, it is easy to see that  $[u \oplus v^{*}](\cd)\in \mathcal{U}[\sigma, T]$ and $\widetilde X(\cd)$ satisfies the following equation: 
							$$
							\left\{\begin{aligned}
								\ds &d \widetilde X(s)=\big[A(s,\a(s)) \widetilde X(s)+B(s,\a(s))[u \oplus v^{*}](s)\big] d s
								+\big[C(s,\a(s)) \widetilde X(s)+D(s,\a(s))[u \oplus v^{*}](s)\big] d W(s), \\
								\ns\ds &\widetilde X(\sigma)= \xi,\q\a(\sigma)=\vf.
							\end{aligned}\right.
							$$
							Furthermore, we have 
							\begin{align}
								&J_0\left(\sigma, \xi,\vf; [u \oplus v^{*}](\cd)\right) \nonumber\\
								&=\mathbb{E}\bigg[\left\langle G(\a(T)) X_{2}^{*}(T), X_{2}^{*}(T)\right\rangle+\int_{\tau}^{T}\left\langle
								\left(\begin{array}{cc}
									Q(s,\a(s)) & S(s,\a(s))^{\top} \\
									S(s,\a(s)) & R(s,\a(s))
								\end{array}\right)\left(\begin{array}{c}
									X_{2}^{*}(s) \\
									v^{*}(s)
								\end{array}\right),\left(\begin{array}{c}
									X_{2}^{*}(s) \\
									v^{*}(s)
								\end{array}\right)\right\rangle d s\bigg] \nonumber\\
								&\q~ +\mathbb{E}\bigg[\int_{\sigma}^{\tau}\left\langle\left(\begin{array}{cc}
									Q(s,\a(s)) & S(s,\a(s))^{\top} \\
									S(s,\a(s)) & R(s,\a(s))
								\end{array}\right)
								\left(\begin{array}{c}
									X_{1}(s) \\
									u(s)
								\end{array}\right),\left(\begin{array}{c}
									X_{1}(s) \\
									u(s)
								\end{array}\right)\right\rangle d s\bigg] \nonumber\\
								&= \mathbb{E}\bigg[V\big(\tau, X_{1}(\tau),\a(\t) \big)
								+\int_{\sigma}^{\tau}\left\langle\left(\begin{array}{cc}
									Q(s,\a(s)) & S(s,\a(s))^{\top} \\
									S(s,\a(s)) & R(s,\a(s))
								\end{array}\right)\left(\begin{array}{c}
									X_{1}(s) \\
									u(s)
								\end{array}\right),\left(\begin{array}{c}
									X_{1}(s) \\
									u(s)
								\end{array}\right)\right\rangle d s\bigg] \nonumber\\
								&= J^\t_0(\sigma, \xi,\vf ; u(\cd)), \label{21.8.22.1.1}
							\end{align}
							where in the last equation we employ relation \rf{21.8.21.6} and the definition of $J^\t_0$. 
							%
							%
							In particular, when $\xi=0$, from \autoref{21.8.21.4} we have
							\begin{equation}\label{22.7.21.1}
								J^\t_0(\sigma, 0,\vf ; u(\cd))=J_0(\sigma, 0,\vf ; [u \oplus v^{*}](\cd)) \geqslant \e \mathbb{E} \int_{\sigma}^{T}\big|\left[u \oplus v^{*}\right](s)\big|^{2} d s \geqslant \e \mathbb{E} \int_{\sigma}^{\tau}|u(s)|^{2} d s,
							\end{equation}
							which combining \autoref{remark5.6}, assertion (ii) of \autoref{Corollary3.5}, and \autoref{Theorem4.2}
							imply assertion (i) holds.
							
							\ms
							
							Now, we consider assertions (ii) and (iii).
							Note that \rf{21.8.22.1.1} still holds when the expectation $\dbE$ replaced by the conditional expectation $\dbE_\sigma[\cd]$. Therefore, 
							\begin{equation}\label{21.8.22.2}
								J^\t(\sigma, \xi,\vf; u(\cd)) \geqslant\big\langle P(\sigma,\vf) \xi, \xi\big\rangle.
							\end{equation}
							For Problem (M-SLQ), assume that $u^{*}(\cd) \in \mathcal{U}[\sigma, T]$ is an open-loop optimal control
							w.r.t. $(\sigma,\xi,\vf)$, and denote by $X^{*}=\left\{X^{*}(s) ; \sigma \leqslant s \leqslant T\right\}$  the related open-loop optimal state process, i.e., for $s \in[\sigma, T],$
							$$
							\left\{\begin{aligned}
								\ds d X^{*}(s) &=\big[A(s,\a(s)) X^{*}(s)+B(s,\a(s)) u^{*}(s)\big] d s
								+\big[C(s,\a(s)) X^{*}(s)+D(s,\a(s)) u^{*}(s)\big] d W(s), \\
								\ns\ds X^{*}(\sigma) &=\xi,\q\a(\sigma)=\vf.
							\end{aligned}\right.
							$$
							Then from \autoref{Corollary4.5}, we see that for Problem (M-SLQ),
							the restriction $\left.u^{*}(\cd)\right|_{[\tau, T]}$ of $u^{*}(\cd)$ on the interval $[\tau, T]$ is optimal for the initial triple $\left(\tau, X^{*}(\tau),\a(\t)\right)$.
							Moreover, in \rf{21.8.22.1.1}, using  $\left.u^{*}(\cd)\right|_{[\sigma, \tau]}$ and $\left.u^{*}(\cd)\right|_{[\tau, T]}$ to replace $u(\cd)$ and $v^{*}(\cd)$, respectively, and
							note that
							$u^{*}(\cd)|_{[\sigma, \tau]} \oplus u^{*}(\cd)|_{[\tau, T]}=u^{*}(\cd),$
							we deduce
							\begin{equation}\label{21.8.22.3}
								J^\t\big(\sigma, \xi,\vf;\left.u^{*}(\cd)\right|_{[\sigma, \tau]}\big)
								=J\left(\sigma, \xi,\vf; u^{*}(\cd)\right)
								=\big\langle P(\sigma,\vf) \xi, \xi\big\rangle .
							\end{equation}
							Therefore, combining \rf{21.8.22.2} and \rf{21.8.22.3}, assertions (ii) and (iii) hold.
						\end{proof}

						\begin{remark}\rm
							Since many notations have been adopted, we finally summarize some frequently used notations so as not to be confused. Recall that the quadruple $(X(\cd),Y(\cd),Z(\cd),\Gamma(\cd))$ is the solution of SDE \rf{state} and BSDE \rf{21.8.20.1},  $(X^*(\cd),Y^*(\cd),Z^*(\cd),\Gamma^*(\cd))$ is the solution of FBSDE \rf{4.1}, and  $(X^*_\t(\cd),Y^*_\t(\cd),Z^*_\t(\cd),\G^*_\t(\cd))$ is the solution of FBSDE \rf{4.1} with $u^*(\cd)$ replaced by $u^*(\cd)|_{[\t,T]}$. Moreover, the quadruple $\big(\boldsymbol{X}(\cd), \boldsymbol{Y}(\cd),\boldsymbol{Z}(\cd),\boldsymbol{\G}(\cd)\big)$ is the solution of FBSDE \rf{22.6.25.2},
							$(\tilde{X}(\cd),\tilde{Y}(\cd),\tilde{Z}(\cd),\tilde \G(\cd))$ is the solution of FBSDE \rf{3.4}, $(\bar{X}(\cd),\bar{Y}(\cd),\bar{Z}(\cd),\bar\G(\cd))$ is the solution of FBSDE \rf{3.5}, and $(\widehat{X}^*(\cd),\widehat{Y}^*(\cd),\widehat{Z}^*(\cd),\widehat{\G}^*(\cd))$ is the solution of FBSDE \rf{4.4}. Besides, the process $\hat P(\cd,\a(\cd))$, together with $\hat\L (\cd)$ and $\hat \z(\cd)$, is the solution of SRE \rf{21.8.23.6}, and $P(\cd,\a(\cd))$ is the process that satisfies the relation \rf{21.8.21.6}.
						\end{remark}

						\section{Example}\label{Example}

						As presented in the introduction, in a realistic market, it is better to allow the market parameters to depend on both the Markov chain and Brownian motion, due to the interest rates, stock rates, and volatilities being affected by the uncertainties caused by the Brownian motion.
						In this section, as an application of our main results, we give an example of the continuous-time mean-variance portfolio selection problem, which partially develops the work of Li--Zhou--Lim \cite{Li-Zhou-Lim 2002} to the Markovian regime switching system with random coefficients.
						For simplicity, we would like to let $m=n=1$.
						
						\ms
						
						Suppose there is a market in which two assets are traded over a finite horizon $[0,T]$. One of the assets is the bond whose asset price $S_{0}(\cd)$ is subject to the following ordinary differential equation:
						$$
						\left\{\begin{array}{l}
							\ds d S_{0}(s)=r(s) S_{0}(s) d s, \quad s \in[0, T], \\
							\ns\ds S_{0}(0)=S_{0}>0,
						\end{array}\right.
						$$
						where $r(\cd)$ is a positive and bounded function, which represents the bond's interest rate.
						The other one of the assets is stock whose asset price $S_{1}(\cd)$ satisfies the following stochastic differential equation:
						$$
						\left\{\begin{array}{l}
							\ds d S_{1}(s)=S_{1}(s)\Big\{b(s,\a(s)) d t
							+ \sigma(s,\a(s)) d W(s)\Big\}, \quad s \in[0, T], \\
							\ns\ds S_{1}(0)=S_{1}>0,\q \a(0)=i_0,
						\end{array}\right.
						$$
						where $b(s,\a(s))$ is the appreciation rate and
						$\sigma(s,\a(s))$ is the volatility or the dispersion of the stock.
						Assume that for any choice of $i\in\cS$, both $b(\cd,i)$ and $\sigma(\cd,i)$ are in the space  $L_{\dbF^W}^\i(0,T;\dbR)$, and there is a positive constant $\d$ such that
						$$
						\sigma(s,\a(s))^2 \geq \delta, \quad \forall s \in[0, T].
						$$
						Assume that the trading of shares takes place continuously and transaction costs and consumption are not considered. Then, a small investor's self-financing wealth process $X(\cd)$ satisfies the following SDE:
						\begin{equation}\label{eq:wealth}
							\left\{\begin{aligned}
								dX(s)
								=&\Big\{r(s) X(s)+[b(s,\a(s))-r(s)] u(s)\Big\} d s
								+\sigma(s,\a(s)) u(s) d W(s),\q s\in[0,T], \\
								X(0)=& x_{0}>0,\q \a(0)=i_0,
							\end{aligned}\right.
						\end{equation}
						where $u(\cd)\in L_\dbF^2(0,T;\dbR)$ is a portfolio of the investor, which may change over time $s\in[0,T]$.
						Note that $u(s)=0$ implies that the investor invests his/her total wealth in the bond at time $s\in[0,T]$.
						
						\ms
						
						As shown in Li--Zhou--Lim \cite{Li-Zhou-Lim 2002},
						the mean-variance portfolio selection refers to the problem of finding an allowable investment policy (i.e., a dynamic portfolio satisfying all the constraints) such that the expected terminal wealth satisfies $\mathbb{E}[X(T)] = d$ while the risk measured by the variance of the terminal wealth is
						$$
						\operatorname{Var}(X(T)) = \mathbb{E}\big[X(T) - \mathbb{E}[X(T)]\big]^{2} = \mathbb{E}\big[(X(T)-d)^{2}\big].
						$$
						Then we consider the following dynamic stochastic optimization
						\begin{equation}\label{eq:mv}
							\left\{\begin{aligned}
								\min ~ & \mathbb{E}\big[(X(T)-d)^{2}\big], \\
								\mbox{s.t. } & \mathbb{E}[X(T)] = d, \\
								& (X(\cdot), u(\cdot)) \mbox{ satisfies (\ref{eq:wealth})}.
							\end{aligned}\right.
						\end{equation}
						Since (\ref{eq:mv}) is a convex optimization problem, the equality constraint $\mathbb{E}[X(T)] = d$ can be dealt with by introducing a Lagrange multiplier $\mu \in \mathbb{R}$. Therefore, we have
						$$
						\mathbb{E}\big[(X(T)-d)^{2}\big] - 2\mu(\mathbb{E}[X(T)] - d)
						= \mathbb{E}\big[(X(T) - (d+\mu))^{2}\big] - \mu^2
						= \mathbb{E}[X(T) - \gamma]^{2} - \mu^2,
						$$
						where $\gamma = d + \mu$.
						Now, if we set $$\widetilde{X}(s) =X(s) - \gamma \exp\Big\{-\int_s^T r(s)ds\Big\},$$ then (\ref{eq:mv}) can be transferred into the following problem
						\begin{equation}\label{eq:mv1}
							\left\{\begin{aligned}
								\ds \min ~ & \mathbb{E}[\widetilde X(T)^2] \\
								\ns\ds \mbox{s.t. } & d\widetilde X(s)
								=\Big\{r(s) \widetilde X(s)+[b(s,\a(s))-r(s)] u(s)\Big\} d s
								+\sigma(s,\a(s)) u(s) d W(s),\q s\in[0,T], \\
								\ns\ds &\ \widetilde X(0) = x_0 - \gamma \exp\Big\{-\int_0^T r(s)ds\Big\},\q \a(0)=i_0.
							\end{aligned}\right.
						\end{equation}
						It is easy to see that
						Problem (\ref{eq:mv1}) is a special case of Problem (M-SLQ) with
						$$G(\a(T)) = 1, \q A(s,\a(s)) = r(s),\q B(s,\a(s)) = b(s,\a(s)) - r(s),\q
						D(s,\a(s)) = \sigma(s,\a(s)),\q s\in[0,T],$$
						and other coefficients are zero.
						Then the SRE associated to Problem (\ref{eq:mv1}) is:
						\begin{equation}\label{21.8.23.6.2}
							\left\{\begin{aligned}
								\ds d\hat P(s,\a(s))&=-\big[\hat{Q}(s,\a(s))
								+\hat{S}(s, \a(s))^\top \Th(s,\a(s))\big] ds\\
								\ns\ds &\q +\hat \Lambda(s) d W(s) +\hat \zeta(s)\bullet d\wt{N}(s), \q s \in[0, T], \\
								\ns\ds \hat P(T,\a(T))&=1, \q \a(0)=i_0,
							\end{aligned}\right.
						\end{equation}
						where $i_0\in\cS$ and for $s\in[0,T]$,
						\begin{equation}\label{hatsq.2}
							\begin{aligned}
								\ds \hat{Q}(s,\a(s))&\deq 2r(s)\hat P(s,\a(s)),\qq
								\hat{S}(s,\a(s))\deq [b(s,\a(s)) - r(s)] \hat P(s,\a(s))
								+\sigma(s,\a(s))\hat \Lambda(s), \\
								\ns\ds \hat{R}(s,\a(s))&\deq \sigma(s,\a(s))^{2} \hat P(s,\a(s)),\qq
								\Theta(s,\a(s))\deq-\hat{R}(s,\a(s))^{-1}\hat{S}(s,\a(s)).
							\end{aligned}
						\end{equation}
						Now, we point out that conditions {\rm (\textbf{H})} and \rf{21.8.21.3.2} hold for this case. In fact, it is trivial to verify condition {\rm (\textbf{H})}. As for condition  \rf{21.8.21.3.2}, by solving the linear SDE of \rf{eq:mv1}, it is easy to verify that there is a positive constant  $\e$ such that
						\begin{align*}
							\dbE[\widetilde X(T)^2]\geqslant\e \mathbb{E}\int_{0}^{T}|u(s)|^{2}ds, \q \forall u(\cd) \in \mathcal{U}[0, T].
						\end{align*}
						Therefore, \autoref{21.9.3.1} and \autoref{21.11.26.1} deduce the following conclusion concerning the mean-variance portfolio selection problem.
						\begin{theorem}\label{22.6.20.1} \sl
							SRE \rf{21.8.23.6.2} admits a unique adapted solution $(\hat P(\cd,,\a(\cd)), \hat \Lambda(\cd),\hat \zeta(\cd)) \in$ $L_{\mathbb{F}}^{\infty}(0, T ; \dbR) \times L_{\mathbb{F}}^{2}(0, T ; \dbR)\times L_{\dbF}^2(0,T;\dbR))$ and
							the mean-variance problem \rf{eq:mv1} is uniquely solvable.
							%
							%
							Moreover, the unique optimal investment strategy $u^{*}(\cd)$ has the following representation:
							\begin{equation}\label{22.6.7.1.2}
								u^{*}(s)=\Theta(s,\a(s)) X^{*}(s), \q s \in[0, T],
							\end{equation}
							where $\Theta(\cd)$ is defined in \rf{hatsq.2}
							and $X^{*}(\cd)$ is the solution of the following closed-loop system:
							\begin{equation*}
								\left\{\begin{aligned}
									\ds &d X^{*}(s)= \Big\{r(s)+[b(s,\a(s)) - r(s)]\Theta(s,\a(s))\Big\} X^{*}(s) d s
									+\sigma(s,\a(s)) \Theta(s,\a(s)) X^{*}(s) d W(s),\ \ s \in[0, T], \\
									\ns\ds &X^{*}(0) =x(0) - \gamma \exp\Big\{-\int_0^T r(s)ds\Big\},\q \a(0)=i_0.
								\end{aligned}\right.
							\end{equation*}
						\end{theorem}

						\begin{remark}\rm
							The above result partially develops the mean-variance problems of Li--Zhou--Lim \cite{Li-Zhou-Lim 2002} to the Markovian regime switching system with random coefficients, which goes beyond the framework of Sun--Xiong--Yong \cite{Sun-Xiong-Yong-21}.
							Note that we consider the one-dimensional state case of two assets (one bond and one stock) to be just for simplicity of presentation, and the multi-dimensional case of $m+1$ assets (one bond and $m$ stocks) can be proved using a similar argument.
						\end{remark}

						\section{Conclusion}\label{Conclusion}
						
						This paper extends the work of Sun--Xiong--Yong \cite{Sun-Xiong-Yong-21} to the framework within the Markovian regime switching system, obtains the solvability of SRE \rf{21.8.23.6} with jumps and random coefficients, gets the closed-loop representation of the open-loop optimal control, and gives a financial application of the continuous-time mean-variance portfolio selection problem, which develops the work of Li--Zhou--Lim \cite{Li-Zhou-Lim 2002}.
						In addition, a new point of view for the uniform convexity of the cost functional is presented, and the equivalence between Problem (M-SLQ)$_0$ and Problem ${(\hb{M-SLQ})}$ is obtained.
						Note that it remains open if someone could get the closed-loop solvability in this model.

						\section*{Acknowledgements}
						
						The authors would like to thank the associate editor and the anonymous referees for their insightful comments that improve the quality of this paper.

					\end{document}